\documentclass{amsart}

\usepackage{tikz,url,mathabx,blkarray,amssymb}

\newcommand{\C}{\mathbb{C}}

\newcommand{\Aut}{{{\text{\rm{Aut}}}}}
\newcommand{\fin}{{\rm f}}

\newcommand{\myuparrow}{\!\uparrow}
\newcommand{\sym}{{\rm Sym}}
\newcommand\myH{\scalebox{0.9}{\bf H}}
\newcommand\myT{\scalebox{0.9}{\bf T}}
\newcommand\myU{\chi_{\scalebox{0.9}{\bf H}}}
\title[Monomial representations in combinatorics]
{Centraliser algebras of monomial representations and applications in combinatorics}

\author[Barrera Acevedo]{Santiago Barrera Acevedo}
\address{
La Trobe University\\
Department of Mathematical and Physical Sciences\\
Bundoora 3083 VIC\\
Australia
}
\email{s.barreraacevedo@latrobe.edu.au}

\author[\'O Cath\'ain]{Padraig \'O Cath\'ain}
\address{
Dublin City University\\
Fiontar agus Scoil na Gaeilge\\
Drumcondra, Dublin 9\\
Ireland
}
\email{padraig.ocathain@dcu.ie}

\author[Dietrich]{Heiko Dietrich}
\address{
Monash University\\
School of Mathematics\\
Clayton 3800 VIC\\
Australia
}
\email{heiko.dietrich@monash.edu}

\author[Egan]{Ronan Egan}
\address{
Dublin City University\\
School of Mathematical Sciences\\
Glasnevin, Dublin 9\\
Ireland
}
\email{ronan.egan@dcu.ie}

\thanks{\'O Cath\'ain acknowledges support from Monash University through the Robert Bartnik Visiting Fellowship; from Technical University of the Shannon; and from the Faculty of Humanities and Social Sciences of Dublin City University.}



\begin{document}

\begin{abstract}
Centraliser algebras of monomial representations of finite groups may be constructed and studied
using methods similar to those employed in the study of permutation groups. Guided by results of
D.\ G.\ Higman and others, we give an explicit construction for a basis of the centraliser algebra of a
monomial representation. The character table of this algebra is then constructed via character sums over double cosets. We locate the theory of group-developed and cocyclic-developed Hadamard matrices within
this framework. We apply Gr\"obner bases to produce a new classification of highly symmetric complex Hadamard matrices.

\end{abstract}

\maketitle

\newtheorem{Tma}{Theorem}[section]

\newtheorem{lemma}[Tma]{Lemma}
\newtheorem{prop}[Tma]{Proposition}
\newtheorem{corollary}[Tma]{Corollary}
\newtheorem{theorem}[Tma]{Theorem}

\theoremstyle{definition}
\newtheorem{defi}[Tma]{Definition}
\newtheorem{example}[Tma]{Example}
\newtheorem{remark}[Tma]{Remark}

\section{Introduction}

Let $\Omega$ be a finite set, $k>0$ an integer, and $\mathcal{B}$ a collection of $k$-subsets of $\Omega$. A permutation $g \in \textrm{Sym}(\Omega)$ acts on the set of all $k$-sets in the natural way: $B\subseteq\Omega$ is mapped to $B^{g} = \{ b^{g} \mid b \in B\}$. The {automorphism group} of $\mathcal{B}$ is the largest subgroup $G\leq \textrm{Sym}(\Omega)$ that satisfies $B^g\in\mathcal{B}$ for all $B\in\mathcal{B}$ and $g\in G$.  Taking $k = 2$, one recovers the automorphism group of a graph, while the automorphism groups of block designs and finite geometries arise by imposing suitable conditions on $\mathcal{B}$. The interplay between algebraic properties of the group $G$ and structural properties of the underlying combinatorial structure is one of the classical topics in algebraic combinatorics.  Graphs, designs, and geometries are described by $\{0, 1\}$-matrices, and so it is natural to consider groups of automorphisms that are  permutation groups. Our work here can be seen as  a generalisation of this theory to combinatorial objects defined over larger~alphabets.

To provide a coherent narrative throughout the paper we focus on $n \times n$ (complex) Hadamard matrices. These have complex entries of norm $1$ and satisfy $MM^{\ast} = nI_{n}$ where $M^\ast$ denotes the conjugate-transpose of $M$ and $I_n$ is the $n\times n$ identity matrix. Hadamard matrices find applications in combinatorics, signal processing, and quantum information theory; see e.g.\ Horadam's book \cite{HoradamHadamard}. In this paper, we combine combinatorics, representation theory, and computational algebra to construct new Hadamard matrices in the centraliser algebra of a suitable group representation. 

A matrix is {monomial} if it contains a unique non-zero entry in each row and column. The automorphism group of a 
Hadamard matrix $M$ is defined to be the group $G$ of all pairs of monomial matrices $(P, Q)$ such that $PMQ^{\ast} = M$. Since $M$ is invertible, we see that $P = MQM^{-1}$, so the projections $\pi_{1}(P, Q) = P$ and $\pi_{2}(P, Q) = Q$ are equivalent representations that satisfy $\pi_1(g)M=M\pi_2(g)$ for all $g\in G$, and so $M$ belongs to the \textit{intertwiner} of $\pi_{1}$ and $\pi_{2}$. In the special case that $M$ satisfies $\pi(g)M=M\pi(g)$ for all $g\in G$ and a representation $\pi$, the matrix $M$ belongs to the \textit{centraliser} of $\pi$. While an 
intertwiner carries the structure of a $\mathbb{C}$-vector space, the centraliser is naturally a $\mathbb{C}$-algebra. 
We show how classical results on associative algebras, induced representations, and character theory of finite groups can be combined to give detailed information about a centraliser algebra, see Section \ref{reptheo} for further details. 
In particular, the eigenvalues of a commutative centraliser algebra are expressed as character sums of $G$. Since 
Hadamard matrices are characterised by norm conditions on their entries and eigenvalues, the property of acting 
by monomial automorphisms on a Hadamard matrix can be reduced to the computation of character sums over certain 
double cosets of $G$ and the solution of a system of norm equations, see Sections \ref{sec:cha} and \ref{sec:Apps} for details. 

Our results do in fact generalise to other matrices that are determined by entry and eigenvalue conditions, e.g.\ weighing matrices, equiangular lines, mutually unbiased bases. The restriction to Hadamard matrices illustrates the complexities of the generalisation from $\{0, 1\}$-matrices to general complex entries. Historically, the natural measures of complexity for Hadamard matrices were the size of the matrix and the size of the set of entries: thus matrices with $k^{\textrm{th}}$ roots of unity as entries were studied extensively. In this paper we propose an alternative measure of complexity: the dimension  of a centraliser algebra containing the Hadamard matrix $M$. Our constructions have locally determined entries: by Proposition \ref{CentMat}, a fixed number of entries in each row of $M$ are related to any given entry by explicit scalar factors. Furthermore the eigenvalues of $M$ are expressed as linear combinations of character sums, see Proposition \ref{prop_CT} and the subsequent discussion. While the computations required are too extensive to carry out by hand except for the smallest cases, advances in computational algebra make these methods practical for matrices with several hundred rows.

\subsection{Main Results}
We outline the structure of this work and, at the same time, highlight some main results. In Section \ref{sec_prelim}, we introduce notation and discuss preliminary results. In Section~\ref{secmon}, we provide a description of the centraliser algebra of a monomial representation that allows for effective computations.

Our results in Section~\ref{sec:grp} concern Hadamard matrices invariant under a group of permutations that acts regularly. More precisely, let $G$ be a group of order $n$ and let $f\colon G\to\C$ be a function. Relative to a fixed ordering of the elements of $G$, define the matrix $M_{f}=[f(gh^{-1})]_{g,h\in G}$. Such a matrix is often called \textit{group-developed} or \textit{group-invariant} in the literature. In Theorem \ref{thm:grp} we show that a matrix $M$ is group-developed over a finite group $G$ if and only if there exist monomial matrices $P, Q$ such that $PMQ^{\ast}$ is in the centraliser algebra of the right regular representation of~$G$. Results of this type are well-known, and real Hadamard matrices that are group-developed are equivalent to Menon difference sets, \cite{BJL, PottFiniteGeometry}.  Our main result in Section~\ref{sec:coc} is to locate the \textit{cocyclic Hadamard matrices} within a similar framework. If $\psi \in Z^{2}(G, \mathbb{C}^{\times})$ is a $2$-cocycle then $M_{\psi} = [\psi(g,h)]_{g,h\in G}$ is a (strictly) cocyclic-developed matrix, \cite{HoradamHadamard, deLauneyFlannery}. In Theorem \ref{thmCHM} we show that a matrix is cocyclic-developed over $G$ if and only if the matrix belongs to the centraliser algebra of a monomial representation associated with the central extension of $G$ determined by $\psi$. One sees that group-developed matrices are cocyclic-developed, and correspond to splitting extensions. Group development involves permutation representations and linear combinations of $\{0,1\}$-matrices. Cocyclic development involves monomial matrices, though it can be related to the group-developed case via group cohomology. 

One goal of this paper is to explain how complex Hadamard matrices with suitable symmetry assumptions may be {constructed} in the centraliser algebra of a monomial representation. Section \ref{sec:development} deals with regular representations, while Section \ref{sec:cha} deals with general representations. The eigenvalues of a matrix that is group-developed over an abelian group are expressible in terms of its characters: continuing the notation of the above paragraph, the matrix $M_f$ can be decomposed as {$M_f=\sum_{g\in G} f(g) L_g$ where each $L_g=[\delta_g^{xy^{-1}}]_{x,y\in G}$}; here $\delta$ is  the usual Kronecker delta. With this definition, $\rho\colon g\to L_g$ is the  left regular representation of $G$. Since $G$ is finite, each $L_g$ is diagonalisable over~$\C$; since $G$ is abelian, all the matrices $L_g$ commute, hence they are simultaneously diagonalisable. It follows that $M_f$ lies in the centraliser algebra of $\rho$ and that the eigenvalues of $M_{f}$ are the Fourier coefficients of $f$, that is, $\sum_{g \in G} f(g)\chi(g)$ where $\chi$ is a (linear) character $\chi$ of $G$.  It is known that a complex $n\times n$ matrix is complex Hadamard if and only if all entries have unit norm and all eigenvalues have complex norm $\sqrt{n}$, see Lemma~\ref{lem:chmnorm}. Thus, the matrix $M_{f}$ is complex Hadamard if and only if the function $f\colon G \rightarrow \mathbb{C}$ has all values of norm $1$ and all of its Fourier coefficients are of norm $\sqrt{n}$. As mentioned earlier, this case is considered further in Section~\ref{sec:grp}. If the group $G$ is nonabelian, then in general one no longer obtains the eigenvalues of $M_{f}$ as explicit functions of the characters of $G$. However, if the monomial representation is multiplicity-free, there exist explicit formulae for the eigenvalues of $M_{f}$ in terms of character sums over certain double cosets, see Proposition \ref{prop_CT}.

The construction of complex Hadamard matrices in the centraliser algebra of a monomial group requires several computational steps. First, given a primitive permutation group, we construct all (perfect) monomial groups supported on this group: under suitable restrictions, this is solved by computing the Schur multiplier. Once we have obtained a monomial group, we construct the \emph{character table} of the centraliser algebra using routines that will be described in Section \ref{sec:cha}. Finally, to construct complex Hadamard matrices, we solve a linear system  $T \underline{\alpha} = \underline{\lambda}$, where $T$ is the character table of the centraliser algebra, $\underline{\alpha}$ is a vector of unknowns (essentially the entries of the resulting matrix) which must have norm~$1$, and $\underline{\lambda}$ is a vector of unknowns (the eigenvalues of the matrix) which must have norm $\sqrt{n}$. In essence, this is a system of quadratic equations. We use Gr\"obner basis routines to construct all solutions. Pursued systematically, this allows us to classify all Hadamard matrices under appropriate symmetry assumptions; we discuss the details in Section \ref{sec:Apps}. To illustrate this approach, the paper concludes in Section~\ref{Sec:comp} with computer constructions of  complex Hadamard matrices that are invariant under a monomial cover of a primitive group of degree at most $15$ and rank $3$. 

{\bf Related work.} Centraliser algebras of induced representations are a well-studied topic in group theory, and for background information we refer particularly to work of Higman \cite{HigmanI,HigmanII, HigmanIII,HigmanIV}, M\"uller \cite{MullerSporadic}, and the textbook of Curtis and Reiner \cite{CurtisReiner}. The theory of cocyclic development has been extensively surveyed; see the monographs of Horadam and de Launey and Flannery, \cite{HoradamHadamard, deLauneyFlannery}. A cohomological approach to some of the results in this paper has been developed independently by Goldberger and collaborators, see \cite{Assaf1,Assaf2, Assaf2024}.

\section{Preliminaries}\label{sec_prelim}
\noindent Unless mentioned otherwise, all groups are finite. 

\subsection{Group actions}\label{Sec:GA}
We refer to \cite{Dixon&Mortimer} for background reading on this section and recall notation here. Let $\Omega$ be a finite set and denote by $\text{Sym}(\Omega)$ the symmetric group on $\Omega$.
A group $G$ acts on $\Omega$ if there is a homomorphism $\pi\colon G\to\text{Sym}(\Omega)$; in this case $\pi$ is called a \textit{permutation representation} of $G$, and its image is a \textit{permutation group on $\Omega$}. We denote the image of $\omega\in\Omega$ under $\pi(g)$ by $\omega.g=\omega.\pi(g)$. The stabiliser  of  $\omega\in \Omega$ is  $G_\omega=\{g\in G: \omega.g=\omega\}$, and the $G$-orbit of $\omega\in \Omega$ is $\omega.G=\{\omega.g:g\in G\}$.

The action is \textit{faithful} if $\ker\pi$ is trivial. The $G$-orbits form a partition of $\Omega$. If $\Omega$ forms a single $G$-orbit then the action is \textit{transitive}; it is  $n$-transitive if the induced action on $n$-tuples over $\Omega$ with pairwise distinct entries is transitive. The action is \textit{semiregular} (or \emph{free}) if every stabiliser is trivial; and \textit{regular} if it is semiregular and transitive. Permutation actions of $G$ on sets $\Omega_{1}$ and $\Omega_{2}$ are \textit{equivalent} if there exists a bijection $\phi\colon\Omega_{1} \rightarrow \Omega_{2}$ such that $\phi(\omega).g = \phi(\omega.g)$ for all $\omega \in \Omega_{1}$ and $g\in G$. There is a bijection between equivalence classes of transitive $G$-actions and conjugacy classes of subgroups of $G$. Specifically, a transitive $G$-action on $\Omega$ is equivalent to the $G$-action on right-cosets of a point stabiliser $G_{\omega}$ via right multiplication.

Let $G$ be transitive on $\Omega$, and let $g\in G$ act on $\Omega\times\Omega$ by $(\alpha, \beta).g = (\alpha.g, \beta.g)$. An \emph{orbital} of $G$ is a $G$-orbit on $\Omega\times\Omega$, and the \textit{rank} of $G$ is the number of orbitals. An orbital $\mathcal{O}$ is \emph{self-paired} if $(\alpha,\beta)\in \mathcal{O}$ whenever $(\beta,\alpha)\in \mathcal{O}$. Since $G$ is transitive on $\Omega$, we can fix $\omega\in\Omega$ such that every $(\alpha, \beta)\in\Omega\times\Omega$ can be written as $(\alpha, \beta)=(\omega.g_{1},\omega.g_{2})$ for  suitable $g_{1}, g_{2}\in G$. As a result, there is a bijection between the orbitals of $G$ and the orbits of $G_\omega$ on $\Omega$ given by the map $(\omega, \omega.g).G \mapsto (\omega.g).G_{\omega}$. The $G_\omega$-orbits on $\Omega$ are \textit{suborbits}, and their cardinalities are the \textit{subdegrees} of $G$; the latter are independent of $\omega$. In addition, there is a bijection between the orbitals of $G$ and the $G_\omega$-double cosets of $G$ given by the maps $(\omega.g_1,\omega.g_2).G\mapsto G_\omega g_2g_1^{-1} G_\omega$ and  $G_\omega g G_\omega \mapsto (\omega,\omega.g).G$; for details and proofs see \cite[Chapter 3]{Wielandt}. Let $H=G_\omega$; the orbital of $G$ corresponding to $HgH$ is denoted $\mathcal{O}_{HgH}$ and has elements $(\omega.k,\omega.gk)$ for $k\in G$. For a set $S$, we define  $\delta_{S}(x)$ to be $1$ if $x\in S$, and $0$ otherwise; the Kronecker delta $\delta_x^y$ is $1$ if $x=y$, and $0$ otherwise.

\subsection{$\mathbb{C}$-algebras and representation theory}\label{reptheo}
Let $\text{M}_n(\mathbb{C})$ be the algebra of $n \times n$ matrices over $\mathbb{C}$, and let $A$ be a $\mathbb{C}$--algebra. An $n$-dimensional {\it representation} of $A$ is an algebra homomorphism $\rho\colon A\rightarrow M_n(\mathbb{C})$. The induced $A$-module structure on the $n$--dimensional row  space $\mathbb{C}^n$ is defined by $v.a=v\rho(a)$ for $v\in\mathbb{C}^n$ and $a\in A$. A representation is \textit{reducible} if there exists a nontrivial submodule, and \textit{irreducible} otherwise. If the associated $A$--module decomposes as a direct sum of irreducible submodules, then the representation is {\it completely reducible}. The \textit{character} of $\rho$ is the trace map $\chi_\rho\colon A\rightarrow \mathbb{C}$,  $a\mapsto\text{Tr}(\rho(a))$; it is called irreducible if and only if $\rho$ is irreducible.

By Maschke's Theorem \cite[Theorem 1.9]{IsaacsRep}, the complex \textit{group algebra} $\mathbb{C}[G]$ is completely reducible. If $\rho = \rho_{1}+\ldots+\rho_r$ is the sum of irreducible representations $\rho_1,\ldots,\rho_r$, then each $\rho_i$ is an \emph{irreducible constituent} of $\rho$, and $r$ is the \emph{rank} of $\rho$. A group algebra representation $\rho\colon \mathbb{C}[G] \rightarrow M_{n}(\mathbb{C})$ restricts to a group homomorphism  $G\to \text{GL}_n(\mathbb{C})$ into the group of invertible complex $n\times n$ matrices; this restriction is an $n$-dimensional (complex) representation of $G$. The irreducible submodules of
a representation of $G$ coincide with those of $\mathbb{C}[G]$. The number of nonisomorphic irreducible representations of $G$ is equal to the number of conjugacy classes in $G$. The centraliser algebra $\text{C}(\rho)$ of $\rho$ is the subalgebra of $M_n(\mathbb{C})$  consisting  of all matrices that commute with every element of $\rho(G)$. Schur's Lemma \cite[Lemma 1.5]{IsaacsRep} states that the centraliser of an irreducible representation consists of scalar matrices. A corollary of this is that $\text{C}(\rho)$ is commutative if and only if each irreducible constituent of $\rho$ occurs with multiplicity $1$, see  \cite[Theorem 1.7.8]{sagan}.

\subsection{Monomial matrices}\label{A-mon}
A matrix $M \in \text{M}_n(\mathbb{C})$ is \emph{monomial} if it has exactly one nonzero element per row and column. A monomial matrix with  entries in $\{0,1\}$ is a \emph{permutation matrix}. A matrix $M$ is monomial if and only if $M = PD$ for a permutation matrix $P$ and diagonal matrix $D$. A permutation representation $\pi\colon G\to\text{Sym}(\Omega)$, over a finite set $\Omega$, yields a representation $\rho\colon G\to \text{GL}_n(\mathbb{C})$ of permutation matrices. We also refer to $\rho$ as a permutation representation; it can be extended to a group algebra representation $\rho\colon \mathbb{C}[G]\to M_n(\mathbb{C})$. A representation $\rho$ is \textit{monomial} if each $\rho(g)$ factorises as $\rho(g)=P_gD_g$ for a permutation matrix $P_g$ and diagonal matrix $D_g$. The \textit{associated permutation representation} is defined by $\pi_\rho(g)=P_g$. By abuse of notation, we say a monomial representation has a permutation group property $\mathcal{P}$ (such as transitive, primitive, etc) if the associated permutation representation has it.

The set of $n \times n$ monomial matrices forms a group under matrix multiplication, and  the direct product of this group with itself acts on $\text{M}_n(\mathbb{C})$ via $(P,Q) \cdot R = PRQ^*$. Two matrices are \emph{equivalent} if they lie in the same orbit, and the \emph{automorphism group} $\Aut(R)$ of $R \in \text{M}_n(\mathbb{C})$ is the stabiliser of $R$ under this action. A subgroup $U \leq \Aut(R)$ acts \emph{regularly} (transitively) if the induced actions on rows and columns are regular (transitive). The \emph{strong automorphism group} $\mathrm{SAut}(R)$ is the subgroup of all $(P,P)\in\Aut(R)$.

\subsection{Gr\"obner bases}\label{sec:GBprelim}

By Hilbert's Nullstellensatz \cite[Section 4.1]{clo}, there is a one-to-one correspondence between ideals in a polynomial ring $\mathbb{C}[x_1, x_2, \ldots, x_n]$ and algebraic varieties in $\mathbb{C}^{n}$. This correspondence is fundamental in algebraic geometry, as it allows the translation of geometric questions into questions about sets of polynomials, for which algorithmic methods are often available. Any serious discussion of these topics would take us far afield from the subject of this paper, we refer the reader to standard references such as Shafarevich \cite{Shafarevich} or Cox, Little \& O'Shea \cite{clo}. A Gr\"obner basis for an ideal is defined with respect to an ordering on the monomials of the polynomial ring, and facilitates computation with the ideal. In particular, the irreducible components of the ideal can generally be read from the Gr\"obner basis without difficulty. Methods to compute Gr\"obner bases are provided by standard computational algebra systems, such as GAP \cite{GAP4} and Magma \cite{MAGMA}. We will be exclusively interested in ideals whose elements are polynomials of degree $2$ such that every indeterminate appears with degree at most $1$, corresponding to solutions of linear systems on which norm conditions are imposed. The next example illustrates our use of Gr\"obner bases.

\begin{example} \label{ex:Grobner4}
Let
\[ M = \begin{pmatrix}
\alpha_1 & \alpha_2 & \alpha_3 & \alpha_4 \\
\alpha_2 & \alpha_1 & \alpha_4 & \alpha_3 \\
\alpha_3 & \alpha_4 & \alpha_1 & \alpha_2 \\
\alpha_4 & \alpha_3 & \alpha_2 & \alpha_1
\end{pmatrix} \quad\text{and}\quad T=\left(\begin{array}{rrrr}
1 & 1 & 1 & 1 \\
1 &-1 & 1 &-1 \\
1 & 1 & -1 &-1\\
1 &-1 & -1 & 1 \\
\end{array}\right),
\]
and suppose we want to determine all complex Hadamard matrices of the form $M$ where $\alpha_1,\ldots,\alpha_4$ are complex units. Note that the matrix $M$ is group-invariant under the Klein four-group $C_2^2$. Here we note without proof that $T$ is the character table of the associated centraliser algebra and the eigenvalues $\lambda_1,\ldots,\lambda_4$ of $M$ are given by the linear system $T \underline{\alpha} = \underline{\lambda}$, where $\underline{\alpha} = (\alpha_{1}, \ldots, \alpha_{4})^{\intercal}$ and $\underline{\lambda} = (\lambda_{1}, \ldots, \lambda_{4})^{\intercal}$; the latter is proved in  Section \ref{sec:grp}. Up to Hadamard equivalence, $\alpha_1$ can be set to $1$. Lemma \ref{lem:chmnorm} shows that the set of tuples $(1, \alpha_2, \alpha_3, \alpha_4) \in \mathbb{C}^{4}$ for which $M$ is complex Hadamard is defined by the norm equations $\alpha_i \alpha_i^\ast = 1$ for $i=2,3,4$ and $\lambda_j \lambda_j^\ast = 2$ for $j=1,\ldots,4$. Since complex conjugation is not $\mathbb{C}$-linear, we introduce variables $\alpha_{ic}$ and $\lambda_{jc}$ denoting the complex conjugates of $\alpha_{i}$ and $\lambda_{j}$ respectively. Thus, the polynomials  in  $\mathcal{R}=\mathbb{Q}[\alpha_2, \alpha_{2c}, \alpha_3, \alpha_{3c}, \alpha_4, \alpha_{4c}]$ describing these conditions are
\begin{align*}
P_2 & = \alpha_2 \alpha_{2c}-1 \hspace*{1cm} P_3  = \alpha_3\alpha_{3c}-1 \hspace*{1cm} P_4  = \alpha_4 \alpha_{4c} - 1 \nonumber\\
Q_{1} & = (1 + \alpha_2 + \alpha_3 + \alpha_4)(1 + \alpha_{2c} + \alpha_{3c} + \alpha_{4c}) - 2 \nonumber\\
Q_{2} & = (1 - \alpha_2 + \alpha_3 - \alpha_4)(1 - \alpha_{2c} + \alpha_{3c} - \alpha_{4c}) - 2  \\
Q_{3} & = (1 + \alpha_2 - \alpha_3 - \alpha_4)(1 + \alpha_{2c} - \alpha_{3c} - \alpha_{4c}) - 2 \nonumber \\
Q_{4} & = (1 - \alpha_2 - \alpha_3 + \alpha_4)(1 - \alpha_{2c} - \alpha_{3c} + \alpha_{4c}) - 2 \nonumber
\end{align*}
These polynomials generate an ideal $\mathcal{I}$ of $\mathcal{R}$.  A Gr\"obner basis for this  ideal  consists of a collection of ideals, each describing one irreducible component of the variety of $\mathcal{I}$. In this case, there are $6$ irreducible components, one of them being the ideal $\mathcal{J}$ generated by $\{\alpha_2 - 1, \alpha_{2c} - 1, \alpha_{3} + \alpha_4, \alpha_{3c} + \alpha_{4c}, \alpha_{4}\alpha_{4c} -1\}.$
Geometrically, $\mathcal{J}$ is a circle, in which $\alpha_{4}$ can be any complex unit, $\alpha_{3} = - \alpha_{4}$, and $\alpha_{1} = \alpha_{2} = 1$. Every point on this circle corresponds to a complex Hadamard matrix when substituted into $M$. The remainder of the Gr\"obner basis consists of five similar ideals obtained by freely permuting $\alpha_{2}, \alpha_{3}, \alpha_{4}$.
\end{example}

We conclude the preliminaries with a relating the entries and eigenvalues of a complex Hadamard matrix, which is an immediate consequence of the Hadamard Inequality.

\begin{lemma}\label{lem:chmnorm}
A complex $n\times n$ matrix $M$ is a complex Hadamard matrix if and only if every entry of $M$ has complex norm $1$ and every eigenvalue of $M$ has complex norm $\sqrt{n}$.
\end{lemma}
\begin{proof}
  If $M$ is a complex Hadamard matrix, then its entries have norm $1$ by definition, and it follows from $MM^*=nI_n$ that every eigenvalue of $M$ has norm $\sqrt{n}$. Conversely, suppose $M$ is an $n\times n$ matrix with the stated properties. Since $M$ has $n$ complex eigenvalues of norm $\sqrt{n}$, its determinant meets the Hadamard bound, that is, $|\det(M)|=n^{n/2}$. It follows that  $MM^*$ is positive definite with diagonal entries of norm $n$ and determinant
  $n^{n}$. A fundamental inequality for positive definite matrices $D=[d_{ij}]_{i,j=1}^n$ is that $|\det(D)| \leq \prod_{i=1}^{n} d_{ii}$, with equality if and only if $D$ is diagonal, see \cite[Theorem 1]{MaxDet} and the discussion afterwards. Thus, $MM^*$ is diagonal. Since every entry of $M$ has unit norm, $M$ is a Hadamard matrix.
\end{proof}

\section{Centralisers of monomial representations}\label{secmon}

In this section, we develop ideas of Higman \cite{HigmanI,HigmanIV} to give an explicit construction for a basis of the centraliser algebra of a monomial representation. This is closely related to the transfer homomorphism of finite group theory, see \cite[Chapter 14]{HallGroups}. Most of our results here follow from the existing literature.

In the following, let $G$ be a finite group with subgroup $H\leq G$ of index $n$. Let $T$ be a right transversal to $H$ in $G$, that is, every $g \in G$ admits a factorisation as \[g =h_{g}t_{g}\]
for uniquely determined $h_{g} \in H$ and $t_{g} \in T$. We define the maps $\myH \colon G \rightarrow H$ and $\myT \colon G \rightarrow T$ by $\myH(g) = h_{g}$ and $\myT(g) = t_{g}$. We let $G$ act on $T$ by setting $t\cdot g=\myT(tg)$ for $t\in T$; this defines a group action.

Let $\chi\colon H\to \mathbb{C}^\times$ be a $1$-dimensional representation of $H$ (commonly referred to in the literature as a linear character), and extend $\chi$ from $H$ to $G$ by
\[
\chi^{+}(g) = \begin{cases}
\chi(g) ~~ &\text{if} ~ g \in H,\\
0 ~~&\text{if} ~ g \not\in H.
\end{cases}
\]
We write $\myU(g) = \chi(\myH(g))$ for the $\chi$-value of the $H$-part of $g$; note that $\myU$ is usually not a homomorphism.

The next proposition gives the construction of the representation of $G$ induced from $\chi$. It is known that every $n$-dimensional transitive monomial representation of $G$ is induced from some $1$-dimensional representation of a subgroup $H$ of index $n$, see \cite[Section 43, Exercise 1]{CurtisReiner1}. To write explicit matrices in the induced representation it is convenient to fix an ordering on $T = \{ t_{1}, \ldots, t_{n}\}$ where $t_{1} = 1$ represents the coset $H$. 

\begin{prop}\label{monrep}
With the previous notation, the monomial representation induced from $\chi$ is the $n$-dimensional representation $\rho_\chi=\chi \myuparrow_H^G$ that maps $g\in G$ to the matrix
\begin{eqnarray}\label{eqInd}
\rho_\chi(g)=\left[\chi^{+}(t_{i}gt_{k}^{-1})\right]_{i,k}.
\end{eqnarray}
\end{prop}

\begin{proof}
Let $g_1,g_2\in G$. The entry in row $i$ and column $k$ in  $\rho_\chi(g_{1})\rho_\chi(g_{2})$ is $\sum\nolimits_{j=1}^n \chi^+(t_ig_1t_j^{-1})\chi^+(t_jg_2t_k^{-1})$.  It is nonzero if and only if $\myT(t_ig_1)=t_j$ and $\myT(t_jg_2)=t_k$. Thus, there is a unique nonzero entry in row $i$, namely,  $\chi^+(t_ig_1t_j^{-1})\chi^+(t_jg_2t_k^{-1})=\chi(t_ig_1g_2t_k^{-1})$ in column $k$ where $t_k=\myT(\myT(t_ig_1)g_2)$. This coincides with the entry in $\rho_\chi(g_1g_2)$ in row $i$ and column $k$, which shows that  $\rho_\chi$ is a homomorphism defining a monomial representation.
\end{proof}

Recall that $g\in G$ acts on $T\times T$ via $(t,s)\cdot g = (\myT(tg),\myT(sg))$. If $M$ is a matrix whose rows and columns are labelled by the ordered set $T=\{t_1,\ldots,t_n\}$, then we denote by $m(t_i,t_j)$ the entry in $M$ in row $t_i$ and column $t_j$. The next proposition shows that a matrix $M$ lies in the centraliser algebra $\text{C}(\rho)$ if and only if for every $G$-orbital $\mathcal{O}$ the elements in $M$ labelled by $\mathcal{O}$ satisfy certain relations.

\begin{prop}\label{CentMat}
With the previous notation, a matrix $M$ with rows and columns labelled by the transversal $T$, is in the centraliser algebra $\text{C}(\rho)$
if and only if for all $g\in G$ and $t\in T$ we have
\begin{equation}\label{eq_cm1}
 m(\myT(g), \myT(t g))= m(1, t) \myU(g)^{-1} \myU(t g).
\end{equation}
 \end{prop}

\begin{proof}
Since $t_{i}gt_{j}^{-1} \in H$ if and only if $t_{j} = \myT(t_{i}g)$ and  $t_igt_j^{-1}=\myH(t_ig)$, we observe
\[\rho(g)M = \left[\sum\nolimits_{j=1}^n\chi^{+}(t_{i}gt_{j}^{-1})m(t_{j},t_{k}) \right]_{i,k} = \left[\myU(t_{i}g)m(\myT(t_{i}g),t_{k}) \right]_{i,k}.\]
Similarly, $t_{j}gt_{k}^{-1}=h \in H$ if and only if $h^{-1}t_j=t_kg^{-1}$, if and only if $t_{j} = \myT(t_{k} g^{-1})$ and $h=\myH(t_kg^{-1})^{-1}$, and so
\[M\rho(g) = \left[\sum\nolimits_{j=1}^nm(t_{i},t_{j}) \chi^+(t_{j}gt_{k}^{-1}) \right]_{i,k} = \left[m(t_{i},\myT(t_{k} g^{-1})) \myU(t_{k}g^{-1})^{-1} \right]_{i,k}.\]
Now $\rho(g)M = M\rho(g)$ holds for all $g\in G$ if and only if for all $g\in G$ and $i,k\in\{1,\ldots,n\}$ we have
\[
m(\myT(t_ig),t_k) =  m(t_i,\myT(t_k g^{-1}))    \myU(t_ig)^{-1}  \myU(t_kg^{-1})^{-1}.
\]
Let $t\in T$ such that $Ht_k=Ht g$, say $t_k=ht g$. This means that $\myH(t_kg^{-1})=h=\myH(tg)^{-1}$, $\myT(tg)=t_k$, and $\myT(t_kg^{-1})=t$, which shows that  $M\in \text{C}(\rho)$ if and only if for all $g\in G$ and  $i \in\{1,\ldots,n\}$ we have
\begin{equation}\label{auxeq1}
m(\myT(t_{i} g), \myT(tg)) = m(t_{i},t)  \myU(t_{i}g)^{-1}  \myU(t g)
\end{equation}
This defines $m(\myT(t_{i} g), \myT(tg))$ in terms of $m(t_{i}, t)$; for $i=1$, we obtain~\eqref{eq_cm1}.
\end{proof}

Note that \eqref{eq_cm1} is well-defined if and only if $\myU(g)^{-1} \myU(t g)=\myU(k)^{-1} \myU(t k)$ whenever $g,k\in G$ satisfy $(\myT(g),\myT(tg))=(u,v)=(\myT(k),\myT(tk))$, and the latter equation holds if and only if  $g,k\in  Hu \cap t^{-1}Hv$. This is captured by the following definition.

\begin{defi}
  Let $G$ be a group with subgroup $H$. Let $\rho = \rho_{\chi}$ be the monomial representation of $G$ induced from a linear character $\chi$ of $H$. Let $T$ be a set of right coset representatives of $H$. The orbital $\mathcal{O}$ associated with $(1,t)$ is \emph{orientable} if and only if for all $(u,v)\in \mathcal{O}$ and $g,k\in Hu\cap t^{-1}Hv$ we have \begin{align}\label{eq_orient}\myU(g)^{-1} \myU(t g)=\myU(k)^{-1} \myU(t k).\end{align}Otherwise, $\mathcal{O}$ is \emph{non-orientable}.
\end{defi}

It follows  that a matrix $M$ in the centraliser algebra must have zero entries at all positions of non-orientable orbitals or, in other words, $M$ must be supported on orientable orbitals. In particular, two matrices supported on the same orbital are linearly dependent, whereas two orbital matrices corresponding to distinct orbitals are linearly independent. In conclusion, the next result follows.

\begin{theorem}\label{thm:basCentAlg}
With the previous notation, the centraliser algebra $\text{C}(\rho)$ has a $\mathbb{C}$-basis spanned by the orientable orbital matrices.
\end{theorem}

\begin{example}
  Let $G=\sym_{n}$ with $n\geq 4$ be the symmetric group of degree $n$, acting on the set of unordered pairs of elements in $\{1, 2, \ldots, n\}$. The point stabiliser $H$ of $\{1,2\}$ is  isomorphic to $\sym_{n-2} \times \sym_{2}$, with  transversal  
  \[ T = \{ 1_{G}, (1,2,i), (2, 1, i), (1,j)(2,k) \mid 3 \leq i \leq n, 3 \leq j < k \leq n \}\,.\] 
  The action has rank~$3$ with orbitals $\mathcal{O}_1=\{\{1,2\}\}$, $\mathcal{O}_2=\{\{1,x\},\{2,x\}\mid  x\ne 1,2\}$, $\mathcal{O}_3=\{\{x,y\}\mid x,y\ne 1,2\}$. Since the commutator subgroup of $H$ has index $4$ in $H$, it follows that  $H$ has three nontrivial linear characters, \cite[(2.23)]{IsaacsRep}. We choose $\chi$ to be the character which has kernel $\sym_{n-2}$, so $\chi(x)$ is nontrivial if and only if the projection onto the direct factor $\sym_{2}$ is nontrivial. We claim that $\mathcal{O}_3$ is non-orientable. To see this, pick $u=1_{G}$ and $t=v=(1,3)(2,4)$, so that  $Hu\cap t^{-1}Hv$ is a subgroup of $\sym_{n}$ isomorphic to $\sym_{2} \times \sym_{2} \times \sym_{n-4}$, fixing $\{1,2\}$ and $\{3,4\}$ setwise. Since $(1,2) \in Hu \cap t^{-1} Hv$, non-orientability is witnessed by
  \[ \myU(u)^{-1} \myU(t u)=1\ne -1= \myU((1,2))^{-1} \myU(t (1,2));\]
indeed, we have   $\myU((1,2))^{-1} \myU(t (1,2))=\chi((1,2))^{-1}\chi((3,4)) = -1$ since $t(1,2)=(1,3,2,4)=(3,4)(1,3)(2,4)$, and   $\myU(u)^{-1} \myU(t u)=1$ holds by definition.
\end{example}

We conclude this section with a convenient test for orientability of orbitals. Given an orbital $\mathcal{O}=(1,t)\cdot G$, the next proposition shows that instead of verifying \eqref{eq_orient} for all $(u,v)\in \mathcal{O}$ and $g,h\in Hu\cap t^{-1}Hv$, it suffices to only consider elements in $H\cap t^{-1}Ht$ as explained below.

\begin{prop}\label{prop:isorbital}
Let $H \leq G$, let $T$ be a right transversal of $H$ in $G$, and let $\chi$ be a linear character of $H$. The orbital $\mathcal{O}$ containing $(1, t)$ is orientable if and only if $\chi(tht^{-1}h^{-1}) = 1$ for all $h \in H \cap t^{-1}Ht$.
\end{prop}

\begin{proof}
Note that $H\cap t^{-1}Ht$ is the stabiliser of $(1,t)$ in $G$.  Let $g,k\in G$ such that $(1,t) \cdot g = (1,t)\cdot k = (u,v)$. Then $g = h_{1}u$, $tg = h_{2}v$, $k = h_{3}u$ and $tk = h_{4}v$ for some $h_1,\ldots,h_4\in H$. It follows that $gk^{-1} = h_{1}h_{3}^{-1}$ and $tgk^{-1}t^{-1} = h_{2}h_{4}^{-1}$. Now suppose that $\mathcal{O}$ is orientable. Since $g,k\in Hu\cap t^{-1}Hv$, the orientability assumption implies that ${\myU}(g)^{-1}{\myU}(tg) =  {\myU}(k)^{-1}{\myU}(tk)$. Using that $\chi$ is a homomorphism on $H$, this can be  rephrased as follows:
    \renewcommand*{\arraystretch}{1.3}
    \[\begin{array}{lrcl}
    &\myU(g)^{-1}{\myU}(tg) &=&  {\myU}(k)^{-1}{\myU}(tk) \\
    \iff& {\myU}(h_{1}u)^{-1}{\myU}(h_{2}v) &=&  {\myU}(h_{3}u)^{-1}{\myU}(h_{4}v) \\
    \iff& \chi(h_{2}h_{4}^{-1}) &=&  \chi(h_{1}h_{3}^{-1}) \\
    \iff& \chi(tgk^{-1}t^{-1}) &=&  \chi(gk^{-1}).
    \end{array}\]
    The stabiliser of $(1,t)$ in $G$ is generated by all $gk^{-1}$ where $g,k\in G$ which satisfy $(1,t)\cdot g=(1,t)\cdot k$, see \cite[Theorem 3.6A]{Dixon&Mortimer}. Thus, the last equation in the above equivalences implies that $\chi(tht^{-1})=\chi(h)$ for all $h\in H\cap tHt^{-1}$.

In the other direction, assume that $\chi(h) = \chi(tht^{-1})$ for all $h \in H \cap t^{-1}Ht$. Suppose that $(1, t)\cdot g= (1, t)\cdot k$ and proceed as before: we can write $g = h_{1}u$, $tg = h_{2}v$, $k = h_{3}u$ and $tk = h_{4}v$ for suitable $h_{i} \in H$, and now  the argument follows from the implications above.
\end{proof}

\section{Locating group-developed and cocyclic  matrices in centraliser algebras}\label{sec:development}

The main results of this section are characterisations of group-developed and cocyclic matrices; we recall the definitions below.  In Theorem~\ref{thm:grp} we see that a matrix is group-developed over a finite group $G$ if and only if there exists an equivalent matrix in the centraliser algebra of the right regular representation of $G$. While this result is well-known, the reader is encouraged to compare this with Theorem \ref{thmCHM}, which shows that a matrix is cocyclic over $G$ if and only if there exists an equivalent matrix in the centraliser  algebra of a certain monomial cover of $G$.

Throughout, let $G$ be a finite group and let $\mathcal{A}$ be a finite (hence cyclic) subgroup of~$\mathbb{C}^\times$.

\subsection{Group-development and permutation representations}\label{sec:grp}

All matrices in this section will have rows and columns labelled by the elements of $G$ with respect to some fixed ordering. A matrix $M$ with entries in $\mathcal{A}$ is called \textit{strictly group-developed over $G$} if there exists a map $f\colon G\rightarrow \mathcal{A}$ such that $M=[f(gh)]_{g,h\in G}$, and \textit{strictly group-invariant} if
$M=[f(gh^{-1})]_{g,h\in G}$, see \cite[Definition~2.17]{HoradamHadamard} and \cite[Definition~10.2.1]{deLauneyFlannery}. These definitions differ by a permutation of columns, hence $M$ is strictly group-invariant if it is $\mathcal{A}$-equivalent to a strictly group-developed matrix, and vice versa. It is convenient to define $M$ to be \textit{group-developed} if it is $\mathcal{A}$-equivalent to a strictly group-developed matrix.

The \textit{right regular representation} of $G$ is defined by $R(g) = [\delta_{y}^{xg}]_{x,y\in G}$ for $g\in G$, where $\delta_a^b$ is the usual Kronecker delta.
Similarly, the left regular representation is defined by $L(g) = [\delta_{y}^{g^{-1}x}]_{x,y\in G}$.
A direct calculation confirms $N=[\delta^{x}_{y^{-1}}]_{x,y\in G}$ satisfies  \begin{align}\label{eq_N1}N^2=I_n\quad\text{and}\quad NR(g)N&=L(g)=L(g^{-1})^\intercal\quad(g\in G),
\end{align}
which shows $L(g)=NR(g)N^*=NR(g)N$ for all $g\in G$, so  $R$ and $L$ are conjugate.

\begin{lemma}\label{lemsgd} With the previous notation, a complex  $n\times n$ matrix $M$ is strictly group-developed over $G$ if and only if $R(g)ML(g)^\intercal = M$ for all $g\in G$.
\end{lemma}
\begin{proof}
  Recall that for $x,y\in G$ we denote by $m(x,y)$ the entry in $M$ in row $x$ and column $y$.  If  $M$ is strictly group-developed over $G$ with map $f\colon G\rightarrow \mathbb{C}$, then $m(x,y)=f(xy)$; if $g\in G$, then $L(g)^\intercal=L(g^{-1})$ implies
\begin{eqnarray*}
R(g)ML(g^{-1})= \left[ \sum_{x, y} \delta^{w g}_{x} f(xy) \delta^{gy}_{z} \right]_{w,z\in G}
 =  \left[ f\left( (w  g)(g^{-1}  z)\right) \right]_{w,z\in G}
 =  M.
\end{eqnarray*}
Conversely, if $M=[m(x,y)]_{x,y\in G}$ satisfies $M=R(g)ML(g)^\intercal=R(g)ML(g^{-1})$ for all $g\in G$, then a  calculation similar to the one before shows that $[m(xg,g^{-1}y)]_{x,y\in G}=[m(x,y)]_{x,y\in G}$ for all $g\in G$. By choosing $g=x^{-1}$ we find that $m(x,y)=m(1,xy)$ for all $y$, and so  $M=[f(xy)]_{x,y\in G}$ where $f\colon G\to \C$ is defined by $f(g)=m(1,g)$.\qedhere
\end{proof}

The next theorem characterises the existence of group-developed matrices.

\begin{theorem}\label{thm:grp}
A matrix with entries in a finite group $\mathcal{A}\leq \mathbb{C}^\times$ is group-developed over $G$ if and only if there exists an $\mathcal{A}$-equivalent matrix in $\text{C}(R)$, where $R$ is the right regular representation of $G$.
\end{theorem}
\begin{proof}
Any group-developed matrix is $\mathcal{A}$-equivalent to a strictly group-developed matrix, so we can assume that  $M$ is a strictly group-developed matrix over $G$. Lemma~\ref{lemsgd} shows that $R(g)ML(g)^\intercal = M$ for all $g \in G$. Thus, $R(g)M=ML(g^{-1})^\intercal$,  and  \eqref{eq_N1}  implies that $R(g)MN=MNR(g)$, hence  $MN \in \mathrm{C}(R)$.

Conversely, let $M\in\text{C}(R)$. Then $R(g)M=MR(g)$ for all $g\in G$, and using $R(g)=NL(g)N$, we obtain $R(g)M=MNL(g)N$. Since $L(g)^{-1}=L(g)^\intercal$, see \eqref{eq_N1}, this is equivalent to  $R(g)MNL(g)^\intercal=MN$. Now Lemma \ref{lemsgd} shows that $M$ is group-developed over $G$, as claimed.
\end{proof}

\subsection{Cocyclic development and monomial representations}\label{sec:coc}
Having discussed the relation between group-developed matrices and permutation representations in the previous section, we now locate cocyclic development in the theory of monomial representations. As described by Goldberger and collaborators \cite{Assaf2}, these ideas lead to a more general theory of cohomology development of matrices.

As before, let $G$ be a finite group and let $\mathcal{A} \leq \mathbb{C}^{\times}$ be a finite subgroup; note that $\mathcal{A}$ is a cyclic group, which we consider a $G$-module with trivial action. Let $\Gamma$ be a central  extension of $G$ by $\mathcal{A}$. By standard theory of group extensions, $\Gamma$ is isomorphic to a group with underlying set $\mathcal{A}\times G$ and multiplication
\begin{align}\label{eqpsi} (a, g)  (b, h) = (ab\psi(g,h), gh)
\end{align}
where $\psi\colon G\times G \rightarrow \mathcal{A}$ is a (normalised) $2$-cocycle, that is, a function satisfying  $\psi(g,1)=\psi(1,g)=1$ and $\psi(g,h)\psi(gh,k) = \psi(g,hk)\psi(h,k)$ for all $g,h,k\in G$; we refer to \cite[Chapter 15]{HallGroups} or \cite[Chapter 12]{deLauneyFlannery} for further details. In the following we write $\Gamma = (G, \mathcal{A}, \psi)$ to record these data about the group extension.

A matrix $M$ with entries in $\mathcal{A}$ is called \textit{strictly cocyclic over $G$} if there exists a cocycle $\psi: G\times G \rightarrow \mathcal{A}$ and a function $\phi: G\rightarrow \mathcal{A}$ such that
\[ M = \left[ \psi(x,y)\phi(xy) \right]_{x,y \in G},\]
where the rows and columns are indexed by the elements of $G$ in a fixed ordering. As for the group-developed case, it is often convenient to consider $\mathcal{A}$-equivalence: the matrix $M$ is \textit{cocyclic over $G$} if it is $\mathcal{A}$-equivalent to a strictly cocyclic matrix.

Motivated by the right and left regular representations of $G$, we define the following monomial representations $R$ and $L$ of $\Gamma$: if $(a,g)$ is an element of the extension $\Gamma$, then
\begin{align}\label{eqRL} R(a, g) = a \left[ \psi(x, g) \delta^{x g}_{y}\right]_{x, y \in G}\quad\text{and}\quad
L( a, g) = a\left[ \psi(g, g^{-1}x) \delta^{x}_{gy} \right]_{x,y \in G}.
\end{align}
That $R$ is a homomorphism follows from the cocycle identity and a  short calculation:
\begin{eqnarray*}
R(a,g) R(b,h) & = & ab \left[ \sum\nolimits_{y\in G} \psi(x, g) \delta^{x g}_{y} \psi(y, h) \delta^{y h}_{z}\right]_{x,z\in G} \\
& = & ab \left[ \psi(x, g) \psi(x g, h) \delta^{x g h}_{z} \right]_{x,z\in G} \\
& = & ab \left[ \psi(x, gh) \psi(g, h) \delta^{x gh}_{z}\right]_{x,z\in G} \\
& = & ab \psi(g,h) \left[ \psi(x, gh) \delta^{x gh}_{z} \right]_{x,z\in G} \\
& = & R(ab \psi(g,h), gh)\,.
\end{eqnarray*}
In applications it will be necessary to consider $L^*$ instead of $L$. The following property holds for $L^*$, which also shows that $L$ is a homomorphism; recall that $\mathcal{A}$ is a finite  subgroup of $\mathbb{C}^\times$, so $a^*=a^{-1}$ for $a\in\mathcal{A}$:
\begin{eqnarray*}
L(a,g)^{\ast} L(b,h)^{\ast} & = & a^{-1}b^{-1} \left[ \sum\nolimits_{y\in G} \psi(g, g^{-1}y)^{-1} \delta^{gx}_{y} \psi(h, h^{-1}z)^{-1} \delta^{hy}_{z}\right]_{x,z\in G} \\
& = & a^{-1}b^{-1} \left[ \psi(g, x)^{-1} \psi(h, gx)^{-1} \delta^{hgx}_{z} \right]_{x,z\in G} \\
& = & a^{-1}b^{-1} \left[ \psi(h, g)^{-1} \psi(hg, x)^{-1} \delta^{hgx}_{z}\right]_{x,z\in G} \\
& = & a^{-1}b^{-1} \psi(h,g)^{-1} \left[ \psi(hg, g^{-1}h^{-1}z)^{-1} \delta^{hgx}_{z} \right]_{x,z\in G} \\
& = & L(ba\psi(h,g), hg)^{\ast}\,.
\end{eqnarray*}
As in the group-developed case, the representations $R$ and $L$ are conjugate. To see this, we define  $N = [\psi(x,x^{-1})\delta_{y^{-1}}^{x}]_{x,y \in G}$ and use the cocycle identity to deduce that
 \begin{align*}
     R(a,g)NL(a,g)^{\ast} &= aa^{-1}\left[\sum\nolimits_{x, y}\psi(w,g)\delta_{x}^{wg}\psi(x,x^{-1})\delta_{y^{-1}}^{x}\psi(g,g^{-1}z)^{-1}\delta_{z}^{gy}\right]_{w,z} \\
     &= \left[\sum\nolimits_{ y}\psi(w,g)\psi(wg,g^{-1}w^{-1})\delta_{y^{-1}}^{wg}\psi(g,g^{-1}z)^{-1}\delta_{z}^{gy}\right]_{w,z}\\
    &=\left[\psi(w,g)\psi(wg,g^{-1}w^{-1})\psi(g,g^{-1}w^{-1})^{-1}\delta_{z^{-1}}^{w}\right]_{w,z} \\
    &=[\psi(w,w^{-1})\delta_{z^{-1}}^{w}]_{w,z \in G} = N.
 \end{align*}
Thus, if $(a,g)\in \Gamma$, then  \begin{align}\label{eq_N2}L(a,g) &= N^*R(a,g)N.
\end{align}

\begin{lemma}\label{lemRMLM}
With the previous notation, let $\Gamma= (G, \mathcal{A}, \psi)$ be a central group extension and $\phi\colon G\to \mathcal{A}$ a map. Then $M = \left[ \psi(x,y) \phi(xy) \right]_{x,y\in G}$ if and only if $R(a,g) M L(a,g)^\ast = M$ for all $(a, g) \in \Gamma$.
\end{lemma}
\begin{proof}
First, assume that $M = \left[ \psi(x,y) \phi(xy) \right]_{x,y\in G}$ is cocyclic. It follows from the cocycle identity that 
\begin{align*}
R(a,g) M L(a,g)^\ast  & =  aa^{-1} \left[ \sum\nolimits_{x, y} \psi(w,g) \delta^{w g}_{x} \psi(x, y) \phi(xy) \psi(g, g^{-1}z)^{-1} \delta^{gy}_{z}\right]_{w, z} \\
& =  \left[ \psi(w, g) \psi(w g, g^{-1}z) \psi(g, g^{-1}z)^{-1}\phi(wz)\right]_{w, z}\\
& = \left[ \psi(w, z) \phi(wz) \right]_{w, z} = M.
\end{align*}
Conversely, assume that $M=[m(x,y)]_{x,y\in G}$ satisfies $R(a,g) M L(a,g)^\ast = M$ for all $(a,g)\in \Gamma$. The chain of equations   
\begin{align*}
R(a,g) M L(a,g)^\ast & = a\left[ \psi(x,g)\delta^{xg}_{y} \right]_{x,y\in G} \left[ m(y,z) \right]_{y,z\in G} a^{-1} \left[ \psi(g,g^{-1}w)\delta^{w}_{gz} \right]_{z,w\in G}\\
& = \left[ \psi(x,g)m(xg,g^{-1}w)\psi(g,g^{1}w)^{-1} \right]_{x,w\in G}
\end{align*}
shows that $\psi(x,g)m(xg,g^{-1}w)\psi(g,g^{-1}w)^{-1} = m(x,w)$. Choose $g=w$, recall that $\psi(g,1)=1$; then $\psi(x,w)m(xw,1)=m(x,w)$. Thus, $M=\left[\psi(x,w)f(xw)\right]_{x,w\in G}$ where $f\colon G\to \mathbb{C}$ is defined by $f(g)=m(g,1)$, which shows that $M$ is cocyclic.  \qedhere
\end{proof}

The main result of this section is the following.

\begin{theorem}\label{thmCHM}
Let $M$ be a square matrix with entries in the finite subgroup $\mathcal{A}\leq \mathbb{C}^\times$ and with rows and columns labelled by a finite group $G$. Then $M$ is cocyclic over $G$ if and only if there exists an $\mathcal{A}$-equivalent matrix in $\text{C}(R)$, where $R$ is the monomial representation \eqref{eqRL} of an extension $\Gamma=\Gamma(G,\mathcal{A},\psi)$ of $G$.
\end{theorem}

\begin{proof}
Any cocyclic matrix is $\mathcal{A}$-equivalent to a strictly cocyclic matrix, so suppose that $M$ is strictly cocyclic over $G$, with extension group $\Gamma = (G, \mathcal{A}, \psi)$. Lemma~\ref{lemRMLM} shows that $R(a,g)ML(a,g)^{\ast}=M$ for all $(a,g) \in \Gamma$. Together with \eqref{eq_N2}, it follows that
\[
R(a,g)MN^{\ast} = MN^{\ast}R(a,g),
\]
and therefore $MN^{\ast} \in \text{C}(R)$. Conversely, let $M \in \text{C}(R)$ for an extension $\Gamma$ and representation $R$ as in the statement. By definition, $R(a,g)M = MR(a,g)$ for all $(a,g) \in \Gamma$, and using $R(a,g) = NL(a,g)N^{\ast}$, we obtain $R(a,g)M = MNL(a,g)N^{\ast}$, hence $R(a,g)MNL(a,g)^{\ast} = MN$.
\end{proof}

Group-developed matrices have constant row  and column sums, since each row and each column is a permutation of the first. If $H$ is a group-developed $n\times n$ Hadamard matrix with row and column sum $s$, then $HJ_n=sJ_n$ and $H^\intercal J_n = sJ_n$ for the $n\times n$ all-$1$s matrix $J_n$. Thus,  multiplying $nI_n=HH^\intercal$ from the right by $J_n$ yields
\[nJ_n = HH^\intercal J_n =s HJ_n = s^2 J_n\]
which forces that $n=s^2$ is a perfect square. This well-known observation restricts the orders at which group-developed Hadamard matrices exist. (Recall that the order of an $n\times n$ Hadamard matrix refers to the dimension $n$.) There are no known restrictions on the orders of cocyclic Hadamard matrices: indeed, it has been conjectured by Horadam \cite[Research Problem 38]{HoradamHadamard} that there exists a cocyclic Hadamard matrix of order $4n$ for all $n$. This conjecture has been verified for all $n <188$. Many constructions of Hadamard matrices are known to be cocyclic, including Sylvester and Paley matrices, \cite{mypaper-nonaffine, ronan}. Some families of cocyclic real and complex Hadamard matrices have also been classified computationally, \cite{mypaper-4p,mypaper-CocyclicComplex}.

\section{Character tables of centraliser algebras}\label{sec:cha}

Recall from Lemma \ref{lem:chmnorm} that a complex Hadamard matrix is characterised by norm conditions on its entries and eigenvalues. Theorem \ref{thmCHM} explains that the existence of a complex Hadamard matrix that is cocyclic with respect to some indexing group can be verified by studying a suitable centraliser algebra of a monomial representation. Theorem \ref{thm:basCentAlg} and Proposition \ref{prop:isorbital} allow us to determine a basis of a centraliser algebra. Thus, to locate complex cocyclic Hadamard matrices, it remains to consider linear combinations of the basis elements of the centraliser algebra and verify the norm conditions. For this the so-called character table of the centraliser algebra will be useful. We discuss this character table here, and focus on the construction of complex Hadamard matrices in  Sections \ref{sec:Apps} and \ref{Sec:comp}.

The representation theory of finite groups is closely related to the representation theory of associative algebras applied to the group algebra $\mathbb{C}[G]$. Several accessible expositions of this theory are available, including the books \cite{AlperinBell,IsaacsRep, LiebeckJames}. We recall that a finite dimensional associative algebra over $\mathbb{C}$ is \textit{semisimple} if its Jacobson radical is trivial, in which case the algebra is a direct sum of simple algebras.

Let $A$ be a finite dimensional semisimple $\mathbb{C}$-algebra. It is well-known that $A$ is a direct sum of matrix algebras, such that the number $r$ of matrix algebras occurring in this direct sum is equal to the number of isomorphism types of simple $A$-modules; see \cite[Lemma 13.14 and Theorem 13.16]{AlperinBell} and the details given in the proofs thereof. Since the centre of each matrix algebra consists of the scalar matrices, it also follows that $r$ is the dimension of the centre of $A$. We denote by $\{M_1,\ldots,M_r\}$ a basis of the centre of $A$ and let $\chi_1,\ldots,\chi_r$ be the irreducible characters of $A$. The \textit{character table} of $A$ is defined to be the $r\times r$ matrix
\[{\rm CT}(A)=\left[\chi_{i}(M_j)\right]_{i,j}.\]
For a representation $\rho$ of a finite group $G$ induced from a linear character $\chi$ of a subgroup $H$, the character table of the centraliser algebra $\text{C}(\rho)$ may be constructed from the character table of $G$, together with some additional data about double cosets of $H$ in $G$. Let $\{t_1,\dots, t_r\}$ be a set of representatives of the $H$-double cosets in $G$. For $i=1,\dots r$, let $H_i=H\cap t_{i}^{-1}Ht_{i}\leq H$ and $k_i=|H:H_i|$. Let $M_{G,H}$ be the matrix that contains the rows of the character table of $G$ corresponding to the irreducible constituents of $\rho$. Let $L$ be a matrix whose rows are indexed by $t_1,\ldots,t_r$, whose columns are indexed by conjugacy classes of $G$, and whose entry $\ell(t_i,C)$ is defined as \[\ell(t_i,C)=\sum_{h\in H} \delta_C(ht_i)\chi (h^{-1})\] where $\delta_C$ is the Kronecker delta for the conjugacy class $C$. While not especially well-known, the following result has appeared in the literature multiple times. A complete proof of the next result is given by M\"uller; closely related results were obtained previously by Tamaschke \cite{tamaschke}, by Higman \cite{HigmanI} and by Curtis-Fossum, \cite{CurtisFossum}. 

\begin{prop}[Proposition 3.20, \cite{MullerDissertation}]\label{prop_CT}
Let $\rho$ be the monomial representation of $G$ induced from a linear character $\chi$ of a subgroup $H\leq G$. Provided that the centraliser algebra $\text{C}(\rho)$ is commutative, its character table is
  \[{\rm CT}(\text{C}(\rho))=\frac{1}{|H|}\cdot M_{G,H}\cdot L^\intercal\cdot{\rm diag}(k_1,\ldots,k_r)\,,\]
with the notation adopted previously.
\end{prop}

The computations required by Proposition \ref{prop_CT} are practical for reasonably sized groups. While non-commutative centraliser algebras are not much more difficult to treat, we restrict to the commutative case in the remainder of this paper. It is well known that  $\text{C}(\rho)$  is commutative if and only if $\rho$ is multiplicity-free, see e.g. \cite[Theorem 1.7.8]{sagan}. The semisimplicity of $\text{C}(\rho)$ implies that $M_1,\ldots,M_r$ are simultaneously diagonalisable, with $r$ common eigenspaces, denoted $V_{1}, \ldots, V_{r}$. Write $\lambda_{i,j}$ for the eigenvalue of $M_{j}$ on the eigenspace $V_{i}$. In this case the character table may be described as
\begin{align}\label{eq_ct}{\rm CT}({\rm \text{C}(\rho)})=[\lambda_{i,j}]_{i,j}\,. \end{align}
Every $M\in\textrm{C}(\rho)$ can be written as $M=\sum_{i=1}^{r} \alpha_{i} M_{i}$ for complex coefficients $\alpha_1,\ldots,\alpha_r$. The eigenvalue of $M$ on the eigenspace $V_i$ is given given by the $i^{\textrm{th}}$ entry of the vector $T\underline{\alpha}$ where $\underline{\alpha}=(\alpha_1,\ldots,\alpha_r)^\intercal$. It follows from Lemma~\ref{lem:chmnorm} that  $M$ is a complex Hadamard matrix if each entry in $\underline{\alpha}$ has norm $1$ and each entry in $T\underline{\alpha}$ has norm $\sqrt{n}$. In other words, there exists a complex Hadamard matrix, say $M$, in $\text{C}(\rho)$ if and only if there is a solution of the system $T\underline{\alpha} = \underline{\lambda}$ where the entries of $\underline{\alpha}$ (the entries of $M$) have norm $1$ and the entries of $\underline{\lambda}$ (the eigenvalues of $M$) have norm $\sqrt{n}$. As discussed in Section \ref{sec:GBprelim}, this yields a system of linear equations and norm equations over a subfield of the complex numbers. While norm equations are not polynomial over $\mathbb{C}$, the system of equations can be rewritten as a real system of quadratic equations in twice as many variables. The next two examples illustrate the construction of complex Hadamard matrices in the centraliser algebra of a permutation group.

\begin{example}
Let $G\leq \sym_{16}$ be the group $G=\langle\sigma,\tau\rangle$ where
\begin{align*} \sigma &= (1, 2)(3, 4)(5, 6)(7, 8)(9, 10)(11, 12)(13, 14)(15, 16),\\\tau&= (2, 3, 5, 9, 16)(4, 7, 13, 8, 15)(6, 11, 12, 10, 14).\end{align*}
The group  $G$ is a Frobenius group of order $80$, with an elementary abelian subgroup of order $16$ and a point stabiliser $H =\langle \tau \rangle$ of order $5$. Let $\rho$ be the permutation representation induced by the trivial character $\chi$ of~$H$. As a list of $H$-double coset representatives in $G$, we choose the identity, $\sigma$, $\sigma^{-1}\tau^{-1}\sigma\tau$, and $\sigma^{-1}\tau^{-2}\sigma\tau^2$.
We denote by $M_1,\ldots,M_4$ the corresponding  orbital matrices; these form a basis for the (commutative) centraliser algebra. We note that $M_1$ is the identity matrix and $M_2,\ldots,M_4$ have constant row sum $5$.  Using  Proposition~\ref{prop_CT}, the character table of the centraliser algebra is 
\[
T =  \begin{blockarray}{cccc}
 M_1 & M_2 & M_3 & M_4\;\;\\ 
 \begin{block}{(rrrr)}
1&   5&   5&   5\;\; \\
   1&  -3&   1&   1\;\;\\
 1&   1&  -3&   1\;\;\\
 1&   1&   1&  -3\;\;\\
 \end{block}\end{blockarray}
\]The eigenvalues of the matrix  $M=M_{1} + M_{2} - M_{3} - M_{4}$ with entries $\{\pm 1\}$ are $-4$ and $4$, so $M$ is a $16\times 16$ Hadamard matrix by Lemma \ref{lem:chmnorm}. More generally, all solutions to $T\underline{\alpha} = \underline{\lambda}$ with $\alpha_{i}\alpha_{ic} = 1$ and $\lambda_{j}\lambda_{jc} = 4$ are obtained via computing the  Gr\"obner basis as described in Section \ref{sec:GBprelim}. All solutions are of the form $\underline{\alpha} = (1,z,-z,-1)$, or a cyclic permutation of the last three coordinates, where $z$ is of norm $1$.
\end{example}

\begin{example} \label{PaleyGraph}
Let $p\equiv 1\bmod 4$ be a prime and denote by $\mathbb{F}_p$ the field with $p$ elements. Let ${\rm AGL}_1(p)$ be the permutation group consisting of affine transformations of  $\mathbb{F}_p$ of the form  $x \mapsto ax + b$, where $a \in \mathbb{F}_{p}^{\times}$ and $b \in \mathbb{F}_{p}$; this group is  $2$-transitive of degree $p$. The index-2 subgroup $G$ of ${\rm AGL}_1(p)$ consisting of transformations $x \mapsto a^{2}x+ b$ has  rank $3$. The stabiliser of a point is cyclic of order $(p-1)/2$, and we consider the permutation representation of degree $p$ induced from the trivial character of the point stabiliser. The centraliser algebra $\text{C}$ of this representation has dimension $3$, and it is well known that one of the basis elements (constructed as in Theorem \ref{thm:basCentAlg}) is an adjacency matrix, $A$, for the so-called \textit{Paley graph}, which has vertices labelled by the elements of $\mathbb{F}_{p}$, and  vertices $x$ and $y$ are connected by an edge if and only if $x-y$ is a quadratic residue in $\mathbb{F}_{p}$, see \cite[Section~7.4.4]{BvM}. 

The Paley graph is regular of degree $k= (p-1)/2$, so $k$ occurs as an eigenvalue of $A$ with multiplicity $1$. By standard results in algebraic graph theory, the other eigenvalues are $\mu = (-1 + \sqrt{p})/2$ and $\nu = (-1 - \sqrt{p})/2$, each with multiplicity $k$. It follows that a basis for the centraliser algebra of $G$ can be constructed as $\{M_1,M_2,M_3\}$, where $M_{1} = I_{p}$, $M_{2} = A$, and $M_{3} = J_p - I_{p} - A$. The all-$1$s matrix $J_p$ has a $1$-dimensional $p$-eigenspace and a $(p-1)$-dimensional $0$-eigenspace. The matrices $J_p,I_p,A$ are simultaneously diagonalisable, so they share common eigenspaces.  This means that the nullspace of $J_p$ is the direct sum of the $\mu$- and $\nu$-eigenspaces of $A$. Since $M_3$ is a linear combination, its value on the $p$-eigenspace of $J_p$ is $p-k-1=k$, on the $\mu$-eigenspace it is $-1-\mu=\nu$, and on the  $\nu$-eigenspace it is  $-1-\nu=\mu$. Thus, the eigenvalues $\lambda_1,\lambda_2,\lambda_3$ of the matrix $M = \alpha_{1}M_{1} + \alpha_{2}M_{2} + \alpha_{3}M_{3}$ are computed as follows:
\[
\begin{pmatrix}
    1 & k & k \\
    1 & \mu & \nu \\
    1 & \nu & \mu \\
\end{pmatrix}
\begin{pmatrix} \alpha_{1} \\ \alpha_{2} \\ \alpha_{3} \end{pmatrix}
= \begin{pmatrix} \lambda_{1} \\ \lambda_{2} \\ \lambda_{3} \end{pmatrix} \,.
\]
Lemma~\ref{lem:chmnorm} shows that $M$ is Hadamard if and only if $\alpha_{i}\alpha_{i}^{\ast} = 1$ and $\lambda_{i}\lambda_{i}^{\ast} = p$ for $i = 1,2,3$. For this consider the rational polynomial ring $\mathcal{R} = \mathbb{Q}[p, \mu, \alpha_{2}, \alpha_{2c}, \alpha_{3}, \alpha_{3c}]$. We construct the polynomials $P_{2},P_{3},Q_{1},Q_{2},Q_{3}$ to encode the norm conditions on $\alpha_{i}$ and $\lambda_{j}$ as in Section \ref{sec:GBprelim}. In addition, we introduce a polynomial $R = (2\mu + 1)^{2} - p$ to encode the relation between $\mu$ and $p$. In this case, the computation yields a number of isolated points, all of which require $p \leq 4$, and one nontrivial component. This component may be parametrised in terms of $\mu = (-1 + \sqrt{p})/2$ as
\[ \alpha_{1} = 1, \quad  \alpha_{2} = \tfrac{ -1 + \sqrt{1-4\mu^2}}{2\mu},\quad\text{and}\quad \alpha_{3} = \alpha_{2c}=\alpha_2^*\,;\]
recall that $p\equiv 1\bmod 4$ is a prime. This solution is unique up to permuting  $\alpha_{2}$ and $\alpha_{3}$, and replacing both by complex conjugates. For example, if $p=5$, then $\alpha_1=1$, $\alpha_2= -1 + \imath \sqrt{ (5 - \sqrt{5})/8}$, and $\alpha_3=\alpha_2^*$.
\end{example}

\section{Constructing complex Hadamard matrices: Schur multipliers}\label{sec:Apps}

Let $M$ be an $n \times n$ complex Hadamard matrix and recall that $\mathrm{SAut}(M)$ is the subgroup of $\Aut(M)$ consisting of pairs $(P,P)$ with $PMP^\ast=M$.  Denote by $\pi_{1}$ and $\pi_{2}$ the projections of $\mathrm{Aut}(M)$ onto the first and second components, respectively, and set $\Gamma = \pi_{1}(\mathrm{SAut}(M))$. Let $\pi\colon\Gamma\to \textrm{Sym}_n$ be the homomorphism that maps a monomial matrix to the induced permutation matrix (identified with a permutation in $\textrm{Sym}_n$). For the rest of this section, we assume that $G = \pi(\Gamma)$ is transitive. We describe relations between the groups $G$, $\Gamma$, and $\mathrm{SAut}(M)$, and the matrix $M$. This is required for Section \ref{Sec:comp} where we start with a permutation group $G$ and construct complex Hadamard matrices such that $\pi\circ\pi_{1}(\textrm{SAut}(M)) \leq G$. 

The next proposition shows that $\textrm{SAut}(M)$ contains a finite subgroup, specified entirely by $G$, which determines the centraliser algebra $\text{C}(\Gamma)$ completely. Here we use the convention that for a matrix group $K$ we denote by $\text{C}(K)$ the centraliser algebra of the identity representation $K\to K$. For an integer $m$ we denote by $\zeta_m$ a primitive complex root of unity; if $K$ is a group, then $K'=[K,K]$ is the commutator subgroup.

\begin{prop} \label{prop:GammaStructure}
 Let $M$ be an $n \times n$ complex Hadamard matrix, and let $\Gamma = \pi_{1}({\rm SAut}(M))$ and  $G = \pi(\Gamma)$ as defined above. Let $\Gamma_\fin=\{L\in \Gamma \mid \det(L)=1\}$. The projection $\Gamma_\fin\to G$ induced by $\pi$ is surjective with cyclic and central kernel $\langle\zeta_nI_n\rangle$, so $|\Gamma_\fin|=n|G|$ is finite.  Moreover, $M\in\text{C}(\Gamma)$ and $\text{C}(\Gamma)=\text{C}(\Gamma_\fin)$.
\end{prop}
\begin{proof}
 By definition, $M\in\text{C}(\Gamma)$. The kernel of $\pi$ consists of diagonal matrices $D={\rm diag}(a_1,\ldots,a_n)$ with $DMD^\ast=M$. Since the entries of $M$ are all nonzero, this forces $a_ia_j^\ast=1$ for all $i,j$, which shows that $\ker\pi$ consists exactly of all scalar matrices whose entries are complex units.

Now consider $L\in\Gamma$. Let $d$ be the (finite) exponent of $G/G'$. Since $\pi$ maps $\Gamma'$ to $G'$, the $d^{\mathrm{th}}$ power of $L$ satisfies $\pi(L^d)\in G'$. In particular, $L^d=SA$ for some scalar matrix $S=\zeta I_n\in\ker\pi$ and a matrix $A\in \Gamma'$. Note that $\det(A)=1$, so $\det(L^d)=\zeta^n$. Set $\sigma = (\det(L)^{1/n})^{\ast}$. Then $\sigma I_{n} \in \ker \pi$ and hence $\sigma L \in \Gamma$. By construction, $\det(\sigma L)=1$ and $\pi(\sigma L)=\pi(L)$. Since $\pi\colon\Gamma\to G$ is surjective, this implies that also the projection $\Gamma_\fin\to G$ is surjective. The elements in the kernel of that projection are scalar matrices of determinant $1$, that is, $\langle\zeta_n I_n\rangle$.

Since $\Gamma_\fin\leq \Gamma$, we clearly have $\text{C}(\Gamma)\leq \text{C}(\Gamma_\fin)$. For the converse, consider $B\in \text{C}(\Gamma_\fin)$. If $L\in\Gamma$, then the previous paragraph shows that  $L=\sigma^{-1} L'$ where $\sigma^{-1} I_n\in \ker\pi$ and $L'\in \Gamma_\fin$. Since $B\in\text{C}(\Gamma_\fin)$, we have  $BL=B\sigma^{-1}L'=\sigma^{-1}BL'=\sigma^{-1}L'B=LB$, that is, $B\in\text{C}(\Gamma)$. Thus, $\text{C}(\Gamma)=\text{C}(\Gamma_\fin)$, as claimed.
\end{proof}

In early work on group representations, Schur studied the following covering problem: given a projective representation $\rho\colon G \rightarrow \mathrm{PGL}_{n}(\mathbb{C})$, construct a group $\hat{G}$ with  representation $\hat{\rho}\colon \hat{G} \rightarrow \mathrm{GL}_{n}(\mathbb{C})$ such that the natural projection of $\mathrm{GL}_{n}(\mathbb{C}) \rightarrow \mathrm{PGL}_{n}(\mathbb{C})$ induces a surjective homomorphism $\pi\colon \hat{G} \rightarrow G$. Recall that a \textit{stem extension} of a finite group~$G$ is a group $S$ containing a central subgroup $L \leq Z(S) \cap S'$ such that $S/L \cong G$.  To solve this problem, Schur introduced what is now known as the \textit{Schur multiplier} of $G$, which is a group isomorphic to  $H^{2}(G, \mathbb{C}^{\ast})$. A \textit{Schur cover} of $G$ is a stem extension of $G$ by its Schur multiplier. A Schur cover is not generally unique up to isomorphism, but this is so if $G$ is perfect (i.e.\ if $G=G'$), see Aschbacher \cite[Section 33]{Aschbacher}. The next result shows that, under suitable hypotheses, $\Gamma$ is determined by a representation of a Schur cover of $G$.

\begin{prop}\label{thm:Schur}
  With the notation of Proposition \ref{prop:GammaStructure}, suppose that $G$ and $\Gamma_\fin$ are perfect. Let $\hat{G}$ be a Schur cover of $G$, let $H\leq G$ be a point stabiliser, and let $\hat{H}\leq \hat{G}$ be the full preimage of $H$ under the projection $\hat{G}\to G$. Then $\Gamma_\fin=\rho(\hat{G})$ for some representation $\rho$ induced from a linear character of $\hat{H}$.
\end{prop}

\begin{proof}
Since $G$ is perfect, the Schur cover $\hat{G}$ is finite, perfect,  unique up to isomorphism, and \emph{universal}, in the sense that the natural projection $\hat{G}\to G$ factors through any other central extension of $G$, \cite[(33.1)--(33.4),(33.10)]{Aschbacher}. We saw in Proposition \ref{prop:GammaStructure} that $\Gamma_\fin$ is a central extension of $G$ by a cyclic scalar subgroup $\langle \zeta_n I_n\rangle$. By assumption, $\Gamma_\fin$ is perfect, which implies that $\Gamma_\fin$ is an epimorphic image of $\hat{G}$, see \cite[(33.8)]{Aschbacher}, say with epimorphism $\psi\colon\hat{G}\to\Gamma_\fin$. Since $G$ is by hypothesis a transitive permutation group, its permutation representation is induced from the trivial character of a point stabiliser $H$. Let $L\leq \hat{G}$ be the central subgroup such that $\hat{G}/L\cong G$. By construction, $L\leq \hat{H}$ and $\hat{H}/L\cong H$. Since $L$ is central, it follows that the permutation action of $G$ on left costs of $H$ coincides with the permutation action of $\hat{G}$ on left cosets of $\hat{H}$. Together with \cite[Section 43, Exercise 1]{CurtisReiner1}, this implies that  an $n$-dimensional monomial representation $\rho$ of $\hat{G}$ satisfies  $\pi\circ\rho(\hat{G}) = \pi(\Gamma_\fin)$ if and only if $\rho$ is induced from $\hat{H}$. This holds in particular for the epimorphism $\psi\colon\hat{G}\to\Gamma_\fin$, as claimed.
\end{proof}

The conditions of Proposition \ref{thm:Schur} can be relaxed somewhat. If $G$ is not perfect, a universal central extension does not exist, and a Schur cover is no longer unique up to isomorphism. Without assuming a perfect extension, the possibilities for $\Gamma_\fin$ are classified by the group of $2$-cocycles $Z^{2}(G, \langle\zeta_nI_n\rangle)$. Computational techniques are known for working with matrices developed from cocycles \cite{deLauneyFlannery}, but in the remainder of this paper we introduce a technique using Gr\"obner bases to build complex Hadamard matrices. 

We now provide an example that illustrates how the Schur multiplier arises naturally in the construction of the Paley Hadamard matrices.  Recall that the centraliser algebra of any $2$-transitive permutation matrix group of degree $n$ is $2$-dimensional, so it is spanned by $I_{n}$ and $J_{n}-I_{n}$, where $J_n$ is the all-$1$s matrix. If $M=\alpha I_{n} + \beta J_{n}$ is a Hadamard matrix for complex $\alpha, \beta$ of norm $1$, then $\alpha\alpha^*+(n-1)\beta\beta^*=n$ and $\alpha\beta^*+\beta\alpha^*+(n-2)\beta\beta^*=0$, which implies $n=\alpha\alpha^*-\alpha\beta^*-\beta\alpha^*+\beta\beta^*$, and therefore $n\leq 4$. The action of $\textrm{PSL}_{2}(q)$ on $q+1$ points is $2$-transitive, which implies that for $q>3$ there is no complex Hadamard matrix in the centraliser algebra of the $(q+1)\times (q+1)$ permutation matrix group $\text{PSL}_2(q)$. In contrast to this, the next example shows that $\textrm{SL}_{2}(q)$, considered as a suitable monomial cover of $\textrm{PSL}_{2}(q)$, admits a Hadamard matrix in its centraliser algebra when $q\equiv 3\bmod 4$.

\begin{example}[Paley I Hadamard matrices]\label{ex:PaleyMatrix}
Let $q\equiv 3\bmod 4$ be a prime power and consider $G = \textrm{PSL}_{2}(q)$ as a $2$-transitive permutation matrix group of degree $q+1$. The Schur cover of $G$ is isomorphic to $\textrm{SL}_{2}(q)$ for all $q > 3$, see \cite[Theorem~7.1.1]{Karp}. Write $\mathbb{F}_q$ for the finite field with $q$ elements, and let
\[ H = \left\{ \begin{pmatrix} a & b \\ 0 & a^{-1} \end{pmatrix}  \mid a , b \in \mathbb{F}_{q}, a \neq 0\right\}\,\]be the stabiliser in $\textrm{SL}_2(q)$ of a $1$-dimensional subspace. Let $\chi$ be the quadratic character of $\mathbb{F}_q$, that is, $\chi(a)=1$ if $a\in \mathbb{F}_q$ is a nonzero quadratic residue, $\chi(a)=-1$ if $a\in \mathbb{F}_q$ is a nonzero quadratic non-residue, and $\chi(0)=0$. By abuse of notation define \[\chi \begin{pmatrix} a & b \\ 0 & a^{-1} \end{pmatrix} = \chi(a);\] this is easily seen to be a character of $H$. The induced representation $\rho$ of $\mathrm{SL}_{2}(q)$ has a centraliser algebra of rank $2$. We now show that this centraliser algebra contains the $(q+1)\times (q+1)$ Paley I Hadamard matrix which is defined as
\[P=\begin{pmatrix} 1 & 1 \\ -1  &Q+I \end{pmatrix}\quad\text{where}\quad Q=\begin{pmatrix} \chi(i-j)\end{pmatrix}_{i,j\in {\mathbb{F}_q}}.\]
We start by showing that \[ T = \left\{ \begin{pmatrix} 1 & 0 \\ 0 & 1 \end{pmatrix} \right\} \cup
\left\{ t_x=\begin{pmatrix} 0 & -1 \\ 1 & x \end{pmatrix} \mid x \in \mathbb{F}_{q}\right\} \,,\]
is a transversal to $H$ in $\textrm{SL}_2(q)$; this follows because every $m=\begin{pmatrix} a& b \\ c & d \end{pmatrix}$ with $c\ne 0$ can be written uniquely as $m=ht$ where
\[h=\begin{pmatrix} ac^{-1}d - b & a \\ 0 & c \end{pmatrix}\in H \quad\text{and}\quad t=t_{c^{-1}d}\in T.\]

We now want to construct a matrix $M$ with entries $m(s,t)$, $s,t\in T$, that lies in the centraliser algebra of $\rho$ and equals $P$.  Proposition \ref{CentMat} shows that we require $m(\myT(g),\myT(tg))=m(1,t)\chi_{\myH}(g)^{-1}\chi_{\myH}(tg)$ for every $g\in G$ and $t\in T$; here we write $1\in T$ for the identity matrix. First we consider  $g=h\in H$. If $t\in T$, then
\[m(1,\myT(th))= m(\myT(h), \myT(th)) = m(1, t) \chi_{\myH}(h)^{-1} \chi_{\myH}(th)\,. \]
If  $x\in\mathbb{F}_q$ and $h=\begin{pmatrix} a & b \\ 0 & a^{-1} \end{pmatrix}\in H$, then \[ t_xh = \begin{pmatrix} a^{-1} & 0 \\ 0 & a \end{pmatrix} \begin{pmatrix} 0 & -1 \\ 1 & a^{-1}(b+xa^{-1}) \end{pmatrix},\]
which implies that  $\chi_{\myH}(t_xh) =\chi(a^{-1})$; therefore $\chi_{\myH}(h)^{-1} \chi_{\myH}(t_xh)=\chi(a^{-1})^2=1$. If we define $m(1,t_x)=1$ for all $x\in \mathbb{F}_q$, then it follows that the first row of $M$ only has entries $1$. To see that $m(t_i,1)=-1$ for $i\in\mathbb{F}_q$, choose $g=t_i$ and $t=t_0$; then $tg\in H$, and since $p\equiv 3\bmod 4$, we obtain
\[m(t_i,1)=m(\myT(g),\myT(tg))=m(1,t_0)\chi_{\myH}(t_i)^{-1}\chi_{\myH}(tg)=\chi(tg)=\chi(-1)=-1.\]

Next, we show the bottom right $q\times q$ block of $P$ is $Q=(\chi(i-j))_{i,j\in\mathbb{F}_q}$. For this let $i,j\in\mathbb{F}_q$ be distinct and consider \[g = \begin{pmatrix} (i-j)^{-1} & j(i-j)^{-1}) \\ 1 & i \end{pmatrix}\in G.\]  The  factorisations
\[ g = \begin{pmatrix}1 & (i-j)^{-1} \\ 0 & 1 \end{pmatrix} \begin{pmatrix} 0 & -1 \\ 1 & i\end{pmatrix}\quad\text{and}\quad t_0g = \begin{pmatrix}i-j & -1 \\ 0 & (i-j)^{-1} \end{pmatrix} \begin{pmatrix} 0 & -1 \\ 1 & j\end{pmatrix}\,\]
show $m(t_i,t_j)=m(\myT(g), \myT(t_0g)) =m(1,t_0) \chi_{\myH}(g)^{-1} \chi_{\myH}(t_0g) = \chi(i-j)$, as claimed.
\end{example}

In the previous example, we explicitly computed a basis $\{I_{q+1}, M\}$ for the centraliser algebra of $\textrm{SL}_{2}(q)$ acting as a monomial group on $q+1$ points. It is also possible to apply the method of Section \ref{sec:cha}. One finds that the character table is equal to
\[ T=\begin{pmatrix} 1 & g \\ 1 & -g \end{pmatrix}\quad\text{where}\quad g=\sum_{j \in \mathbb{F}_{q}^{\ast}} \chi(j)e^{2\pi\imath \text{Tr}(j)/p}.\]
is a Gauss sum (see \cite[Section 6.3]{IrelandRosen}), where $\chi$ is the quadratic character on $\mathbb{F}_{q}^{\ast}$ and $\text{Tr}$ is the field trace to the prime field. When $q \equiv 3 \bmod 4$ the sum is evaluated as $g=\imath \sqrt{q}$, see \cite[Proposition~6.3.2]{IrelandRosen}. Hence, $(g + 1)(g^{\ast} + 1) = q+1$, and the existence of the Paley I matrices is immediate from Lemma \ref{lem:chmnorm}. When $q \equiv 1 \bmod 4$, the Gauss sum $g=\sqrt{q}$ is real and occus as an eigenvalue of the symmetric matrix 
\[ M = \begin{pmatrix} 0 & \textbf{1} \\ \textbf{1} & M_{2} - M_{3} \end{pmatrix} \]
where $M_{2}$ and $M_{3}$ are as in Example \ref{PaleyGraph}. A similar computation shows that $M_q=I_{q+1} + \imath M$ is complex Hadamard. These matrices are closely related to the Paley II (real) Hadamard matrices of order $2q+2$; for a fuller discussion see \cite{mypaper-morphisms}. 

When $q \equiv 1 \bmod 4$, the group $\textrm{SAut}(M_{q})$ contains both a monomial cover of $\textrm{SL}_{2}(q)$ and the scalar subgroup $\langle (\imath I_{q+1},\imath I_{q+1})\rangle$, see \cite[Section 17.2]{deLauneyFlannery}. Since $q+1 \equiv 2 \bmod 4$, the determinant of $\imath I_{q+1}$ is $-1$. Since $\textrm{SL}_{2}(q)$ is perfect, it follows from the definition of the Schur multiplier that the monomial preimage of $\textrm{SL}_{2}(q)$ in $\Gamma$ is contained in the commutator subgroup $\Gamma'$, and hence cannot contain $\pm \imath I_{q+1}$. Thus the automorphism group contains a subgroup isomorphic to a central product of $C_{4}$ and $\textrm{SL}_{2}(q)$, intersecting in a cyclic subgroup of order $2$. In fact, the full automorphism group is obtained by allowing field automorphisms to act entrywise, see \cite[Section 17.2]{deLauneyFlannery}.

We have shown that if $M$ is a complex Hadamard matrix with $\pi_{1}(\textrm{SAut}(M)) = \Gamma$, then $M\in \mathrm{C}(\Gamma)$. The next example shows that this conclusion no longer holds when the strong automorphism group is replaced with the (ordinary) automorphism group.
It may happen that $\Gamma_{1} = \pi_{1}(\textrm{Aut}(M))$ and $\Gamma_{2} = \pi_{2}(\textrm{Aut}(M))$ are induced from non-conjugate subgroups, in which case $M$ belongs to the \textit{intertwiner} of distinct representations rather than a centraliser algebra. This is illustrated in the next example.

\begin{example}\label{Sylvester}
 The Sylvester Hadamard matrix of order $2^{n}$ can be defined as $S_{n} = [(-1)^{x^{\intercal}y}]_{x,y \in \mathbb{F}_{2}^{n}}$ where $\mathbb{F}_2^n$ denotes the space of $n$-dimensional column vectors over $\{0,1\}$.   Moreover, for $n \geq 2$ it is known that
\begin{equation}\label{AutSn}
\Aut(S_{n}) \cong Z(\Aut(S_{n})) \times (C_{2}^{n} \rtimes \mathrm{AGL}_{n}(2)),
\end{equation}
where $C_2^n$ is the $n$-fold direct product of the cyclic group of size $2$ and $ Z(\Aut(S_{n})) = \langle (-I_{2^n},-I_{2^n}) \rangle$, see \cite[Theorem 9.2.4]{deLauneyFlannery}. For $x,y\in \mathbb{F}_2^n$ denote by $r_{x}$ and $c_{y}$ the row and column of $S_n$ labelled by $x$ and $y$, respectively. The action of $\mathrm{AGL}_{n}(2)$ on rows and columns is described in detail in \cite{mypaper-explicitmorphisms}: if $(v,A)\in \mathrm{AGL}_{n}(2)$ is the transformation $x\mapsto Ax+v$, then
\[
r_{x}\cdot(v,A) = r_{Ax + v} \quad \text{and} \quad c_{y}\cdot(v,A) = (-1)^{v^{\intercal} (A^{-1})^\intercal y}c_{(A^{-1})^{\intercal}y}.
\]
Observe that the action on rows is a $2$-transitive permutation action, whereas the column $c_{0}$ is stabilised. Let $V \leq \textrm{AGL}_{n}(2)$ be the subgroup of translations, and observe that $\pi_{1}(V)$ is a regular permutation group while $\pi_{2}(V)$ is trivial. Hence, $S_{n}$ does not belong to the centraliser of $\textrm{AGL}_{n}(2)$, although it \emph{does} admit an action of a $2$-transitive permutation group. In fact, the stabiliser of a row in $\Aut(S_{n})$ is not conjugate to the stabiliser of a column (the projections onto $\textrm{GL}_{n}(2)$ are the stabiliser of a point and of a hyperplane, respectively). Hence the actions on rows and columns are linearly equivalent, but not monomially equivalent, and $S_{n}$ belongs to the intertwiner of these representations, but not to the centraliser algebra of either of these representations. Specifically, $\pi_{1}(\Aut(S_{n}))$ and $\pi_{2}(\Aut(S_{n}))$ are equivalent as monomial representations,  but the permutation representations $\pi(\pi_{1}(\Aut(S_{n})))$ and  $\pi(\pi_{2}(\Aut(S_{n})))$ are inequivalent as the former is the the regular representation of $V$ while the latter  is trivial.
\end{example}

A similar phenomenon where inequivalent permutation representations lift to equivalent monomial representations can be used to construct complex Hadamard matrices of orders $6$ and $12$ related to the outer automorphisms of $\textrm{Sym}_{6}$ and the Mathieu group $M_{12}$ respectively, \cite{mypaper-S6, HallM12}.

\section{Complex Hadamard matrices admitting a low rank automorphism group}\label{Sec:comp}

In Example \ref{ex:PaleyMatrix}, we saw that for $q \equiv 3 \bmod 4$ the group $\textrm{SL}_{2}(q)$ is isomorphic to a subgroup of the strong automorphism group of the $(q+1)\times (q+1)$ Paley type I matrix. It is natural to ask which other low-rank permutation groups act as the group of strong automorphisms of a complex Hadamard matrix: we recall results in this direction by Moorhouse and Chan, and then describe the results of a computer classification. Detailed information about the $2$-transitive permutation groups suffices to carry out this programme.

\begin{theorem}[Moorhouse,\!\cite{Moorhouse}]\label{thm:Moorhouse}
Suppose that $M$ is a complex Hadamard matrix, and that $G$ is a $2$-transitive permutation group contained in $\pi \circ\pi_{1}(\Aut(M))$. Then one of the following occurs.
\begin{enumerate}
    \item[\rm (1)] $G \cong \mathrm{AGL}_{n}(p)$ in its natural action on $p^{n}$ points, and $M$ is a generalised Sylvester matrix (the character table of an elementary abelian $p$-group),
    \item[\rm (2)] $G \cong \mathrm{PSL}_{2}(q)$ acting on $q+1$ points, and $M$ is a Paley matrix of order $q+1$ that is real for $q \equiv 3 \bmod 4$ and over $4^{\rm th}$ roots of unity for  $q \equiv 1 \bmod 4$,
    \item[\rm (3)] $G \cong \mathrm{Sp}_{2d}(q)$ where $q$ is a power of $2$ and $q^{2d} \geq 16$, and $M$ is of order $q^{2d}$,
    \item[\rm (4)] $G$ is isomorphic to one of ${\rm Alt}_{6}, M_{12}, {\rm P}\Sigma{\rm L}_{2}(8)$ or ${\rm Sp}_{6}(2)$; and $M$ is of order $6$, $12$, $28$ or $36$ respectively.
\end{enumerate}
\end{theorem}

We emphasise that Moorhouse studies the full automorphism group rather than the group of strong automorphisms. Throughout our classification, we require that $\pi_{1}(\mathrm{Aut}(M))$ and $\pi_{2}(\mathrm{Aut}(M))$ are conjugate not only as linear representations, but as monomial representations: that is, the representations are induced from the same subgroup, or equivalently, $M$ belongs to the centraliser algebra of $\pi_{1}(\Aut(M))$. Moorhouse does not make this assumption: he allows representations induced from non-conjugate subgroups, equivalently $M$ may belong to the intertwiner of representations which are not monomially equivalent. The Sylvester matrices  are an example of this phenomenon, see Example \ref{Sylvester}.

Complex Hadamard matrices in the centraliser algebra of a strongly regular graph  have been considered by Chan and Godsil \cite{Chan12,GodsilChan}. Recall that a $(v,k,\lambda, \mu)$-strongly regular graph is a $k$-regular graphs on $v$ vertices in which any two adjacent vertices share $\lambda$ common neighbours, while any pair of non-adjacent vertices share $\mu$ neighbours. The centraliser algebra of a rank $3$ permutation group of even order is spanned by the adjacency matrix of a strongly regular graph, see \cite[Section 1.1]{BvM}. Not every strongly regular graph admits a rank $3$ group action, thus we only state the following special case of relevance to our purposes.

\begin{theorem}[Chan, Godsil, \cite{Chan12, GodsilChan}]\label{thm:Chan}
Suppose that $M$ is a complex Hadamard matrix, and that $G$ is a rank $3$ permutation group contained in $\pi \circ\pi_{1}({\rm SAut}(M))$. Then one of the following holds:
\begin{enumerate}
    \item[\rm (1)] $n = 4t^2$ for an integer $t$, and $G$ {is a group of automorphisms of} a $(4t^2,2t^2 - t, t^2 - t, t^2 - t)$-strongly regular graph (equivalently a Menon Hadamard design).
    \item[\rm (2)] $n = 4t^2 - 1$ for an integer $t$, and $G$ {is a group of automorphisms of}  a $(4t^2 -1,2t^2, t^2, t^2)$-strongly regular graph.
    \item[\rm (3)] $n = 4t^2 +4t + 1$ and $G$ {is a group of automorphisms of}  a $(4t^2 + 4t + 1,2t^2 + 2t, t^2 + t -1, t^2 + t)$-strongly regular graph; here either $t$ or $t^2+t$ is an integer.
    \item[\rm (4)] $n = 4t^2 +4t + 2$ for an integer $t$,  and $G$ {is a group of automorphisms of}  a $(4t^2 + 4t + 2,2t^2 + t, t^2 -1, t^2)$-strongly regular graph.
\end{enumerate}
\end{theorem}

Recall that the subdegrees of a transitive permutation group are the lengths of the orbits of a point stabiliser. These subdegrees are given in the row of the character table of the centraliser algebra corresponding to the trivial irreducible character. In fact, for the groups related to strongly regular graphs considered above, it follows from \cite[Section 1.1.4]{BvM} that the  subdegrees determine all remaining entries of the character table. Thus, Theorem~\ref{thm:Chan} gives a condition on the subdegrees of a rank $3$ permutation matrix group $G$ which is necessary and sufficient for $\text{C}(G)$ to contain a complex Hadamard matrix. A classification of rank 3 permutation groups, including their subdegrees is available in the literature, see for example\cite{LiebeckSaxlRank3}. Thus, while we do not give the classification explicitly here, the classification is in principle known. Chan's classification applies only to rank $3$ permutation matrix groups in which all orbitals are self-paired (that is,  the centraliser algebra has a basis of symmetric matrices). This omits, for example, the Frobenius groups of order $\binom{p}{2}$ where $p \equiv 3 \bmod 4$; such groups have been considered previously Munemasa and Watatani, and independently by Nu\~nez Ponasso, \cite{MunemasaWatatani, PonassoThesis}.

Apart from these results, the literature on classifying complex Hadamard matrices by their automorphism groups is rather sparse. We note substantial work on the closely related problem of classifying complex Hadamard matrices in association schemes by Ikuta and Munemasa, see e.g. \cite{IkutaMunemasa1, IkutaMunemasa2} and references therein. It appears that there is only a single real Hadamard matrix in the literature which admits a primitive-but-not-2-transitive automorphism group. This matrix has  order $144$ and was described by Marshall Hall in \cite{HallHadamard}. To our knowledge, monomial covers of rank $3$ permutation groups have not been investigated, nor have groups of higher rank.

\subsection{Computational classification results}

We conclude this paper with some computational results, building on the theory developed thus far. Given a transitive permutation group $G\leqslant \text{Sym}(\Omega)$ and point stabiliser $H=G_\omega$ (with $\omega\in \Omega$), our algorithm proceeds as follows.
\begin{enumerate}
\item Construct a Schur cover $\hat{G}$ for $G$, and compute the full preimage $\hat{H}\leq \hat{G}$ of $H\leq G$.
\item For each linear character $\chi$ of $\hat{H}$, compute the character of the induced representation $\rho =\chi\myuparrow_{\hat{H}}^{{\hat{G}}}$, and compute the character table $T=\textrm{CT}(\rho)$ of its centraliser algebra  via Proposition \ref{prop_CT}. Denote by $\mathbb{K}$ the field of definition of $T$, and by $r$ the number of rows in $T$.
\item Proceed as in Example \ref{ex:Grobner4}: Define $K = \mathbb{K}[\alpha_{1}, \alpha_{1c}, \alpha_{2},\alpha_{2c}, \ldots, \alpha_r,\alpha_{rc}]$ and construct a Gr\"obner basis of the ideal $\mathcal{I}$ generated by the polynomials that encode the norm conditions for $\alpha_1,\ldots,\alpha_r$ to define a complex Hadamard matrix.  The result is an ideal (defined over $\mathbb{K}$) in which variables are eliminated according to a monomial ordering; due to the structure of our original polynomial equations, there exists a polynomial in the Gr\"obner basis that expresses one of the variables  in terms of a univariate polynomial.
\item Solve for the roots of a univariate polynomial in the Gr\"obner basis; for each solution, substitute the values in the remaining polynomials and then iterate this process. This way it is possible to find all points in the variety.
\item If $(\alpha_1,\alpha_{1c},\ldots,\alpha_r,\alpha_{rc})$ is one of the points in the variety, then this defines a complex Hadamard matrix provided that $\alpha_{ic}$ is the complex conjugate of $\alpha_i$ for all $i$; the latter test is still necessary because it does not follow from the imposed condition $\alpha_i\alpha_{ci}=1$. Once this is verified, the resulting complex Hadamard matrix may be constructed explicitly via Proposition~\ref{CentMat}.
\end{enumerate}

If $G$ and $\Gamma_{f}$ are perfect, then, by Proposition \ref{thm:Schur}, the above algorithm produces all complex Hadamard matrices such that $\pi\circ\pi_{1}({\rm SAut}(M)) \cong G$. When these hypotheses do not hold, the algorithm produces matrices, but without a guarantee of completeness. This procedure requires some heavy machinery: computation of Schur multipliers is a notorious problem in group theory, naive implementations of Proposition \ref{prop_CT} require an iteration over all elements of $G$, and the complexity of computing a Gr\"obner basis is well-known to be doubly exponential in the number of variables. Nevertheless, the algorithm seems practical for permutation groups of order $\leq 10^8$ and of rank $\leq 5$. For the convenience of the reader we illustrate this approach with an explicit example; this matrix has been described by Nuñez Ponasso and by Munemasa-Watani, \cite{MunemasaWatatani, PonassoThesis}.

\begin{example}
There is, up to conjugacy, a unique group $G=C_7\rtimes C_3$ that acts transitively on $7$ points, with point stabiliser $H\cong C_3$. Let $\chi$ be a non-trivial character of $H$, with induced monomial representation $\rho$. Let $\omega=\zeta_3$ be a primitive $3^{\mathrm{rd}}$ root of unity. The centraliser $\text{C}(\rho)$ is spanned by $\{I_{7},M_1,M_2\}$ where
\[
M_{1} = \begin{pmatrix}
0 & 1 & \omega^{2} & \omega &0 & 0 & 0 \\
   0 & 0 &  \omega^{2} &0 & 0 &  \omega & \omega^{2} \\
   0 & 0 & 0 &  \omega^{2} & \omega & \omega^{2} &0 \\
   0 &  \omega^{2} &0 & 0 &  \omega^{2} &0 &  \omega \\
  1 &    1 &0 & 0 & 0 &  \omega &0 \\
    \omega^{2} &0 & 0 &1 &0 & 0 &  \omega \\
    \omega &0 &1 &0 &  \omega &0 & 0\end{pmatrix}\quad\text{and}\quad
M_{2} = \begin{pmatrix}
0 & 0 & 0 & 0 &1 & \omega & \omega^{2} \\
  1 &0 & 0 &  \omega &    1 &0 & 0 \\
    \omega & \omega &0 & 0 & 0 & 0 &1 \\
    \omega^{2} &0 &  \omega &0 & 0 &1 &0 \\
   0 & 0 &  \omega^{2} & \omega &0 & 0 &  \omega^{2} \\
   0 &  \omega^{2} & \omega &0 &  \omega^{2} &0 & 0 \\
   0 &  \omega &0 &  \omega^{2} &0 &  \omega^{2} &0\end{pmatrix}
\,.\]
These matrices are unique up to conjugation by permutation matrices and multiplication by scalars. It is not immediately obvious whether $G$ acts on a complex Hadamard matrix $M$, that is, whether $\text{C}(\rho)$ contains a complex Hadamard matrix. The latter holds if and only if there exist complex numbers $\alpha_{1}, \alpha_{2}$ of norm $1$ such that $M = I_{7} + \alpha_{1}M_{1} + \alpha_{2} M_{2}$ with  $MM^{\ast} = 7I_{7}$. The character table for $\text{C}(\rho)$ is given below, along with the linear equation corresponding to a complex Hadamard matrix.
\[ \begin{pmatrix}
1 & 1 & 1 \\
3 & \frac{-1-\imath \sqrt{7}}{2} & \frac{-1+\imath \sqrt{7}}{2} \\
3 & \frac{-1+\imath \sqrt{7}}{2} & \frac{-1-\imath \sqrt{7}}{2}
\end{pmatrix}\begin{pmatrix} 1 \\ \alpha_{1} \\ \alpha_{2} \end{pmatrix} = \begin{pmatrix} \lambda_{1} \\ \lambda_{2} \\ \lambda_{3} \end{pmatrix} \]
One solution is given by $\alpha_{1} = (-3+\imath\sqrt{7})/4$ and $\alpha_{2} = 1$; there are three more solutions, obtained by swapping $\alpha_{1}$ with $\alpha_{2}$ and taking complex conjugates. It follows that $G$ does act on complex Hadamard matrices defined over the field $\mathbb{Q}(\zeta_{21})$; the latter is the smallest cyclotomic field containing $\imath\sqrt{7}$ and $\omega$.
\end{example}

In Table \ref{tab1} we report on our findings for primitive permutations groups of degree $n\leq 15$ and of rank $3$; a database of such groups is available in MAGMA \cite{MAGMA}. The notation used in the table is as follows. We denote by $\hat{G}$ a Schur cover of $G$, and $\hat{H}$ is the preimage in $\hat{G}$ of a point stabiliser in $G$. We denote by $d$ the largest value  for which a primitive $d^{\textrm{th}}$ root of unity appears in the character (thus, a ``$1$'' indicates a permutation representation). The next column gives a minimal polynomial for the smallest field extension containing the entries of the complex Hadamard matrix. As before, $C_n$ denotes a cylic group of order $n$, and $\zeta_n$ is a  primitive $n^{\mathrm{th}}$ root of unity.

\begin{table}[h]\small
  \def\arraystretch{1.2}%
    \hspace*{-0.17cm}\begin{tabular}{| c | c | c | c | c | c | c |}
    \hline
$n$     & $G$              & $|\hat{G}|/|G|$     & $\hat{H}/\hat{H}'$   &  $d$   & \textrm{Minimal polynomial over $\mathbb{Q}(\zeta_{d})$} & subdegrees\\ \hline
$7$     & $C_{7}\rtimes C_{3}$  & $1$      & $C_{3}$                   & $d \mid 3$ & $x^{2} + \frac{3}{2}x + 1$ & $[1,3,3]$ \\
\hline
$9$     & $C_{3}^2 \rtimes C_{4}$ & $3$      & $C_{12}$                   & $d \mid 12$  & $x^2 - \frac{1}{2}x + 1$ & $[1, 4, 4]$ \\
\hline
$10$     & ${\rm Alt}_5$ & $2$      & $C_{4}$                   & $1$  & $x^{2} + \frac{1}{2}x + 1$;\quad  $x^{2} + 1$  & $[1, 3, 6]$ \\
$10$     & ${\rm Alt}_5$ & $2$      & $C_{4}$                   & $2$  & $x^{4} - 8x^{2} + 36$ & $[1, 3, 6]$ \\
\hline
$11$ & $C_{11}\rtimes C_{5}$ & $1$ & $C_{5}$ & $d \mid 5$ & $x^{2} + \frac{5}{3}x + 1$ & $[1,5,5]$ \\
\hline
$13$ & $C_{13} \rtimes C_{6}$ & $1$ & $C_{6}$ & $1$ & $x^{4} + \frac{1}{3}x^{3} + \frac{5}{3}x^{2} + \frac{1}{3}x + 1$  & $[1, 6, 6]$ \\
$13$ & $C_{13} \rtimes C_{6}$ & $1$ & $C_{6}$ & $2$ & $x^{4} - \frac{1}{3}x^{3} + \frac{5}{3}x^{2} - \frac{1}{3}x + 1$  & $[1, 6, 6]$ \\
$13$ & $C_{13} \rtimes C_{6}$ & $1$ & $C_{6}$ & $3$ & $x^{8}\!-\!\frac{1}{3}x^{7}\!-\!\frac{14}{9}x^{6}\!-\!\frac{1}{9}x^5\!+\!\frac{5}{3}x^{4}\!-\!\frac{1}{9}x^{3}\!-\!\frac{14}{9}x^2\!-\!\frac{1}{3}x\!+\!1$  & $[1, 6, 6]$ \\
$13$ & $C_{13} \rtimes C_{6}$ & $1$ & $C_{6}$ & $6$ & $x^{8}\!+\!\frac{1}{3}x^{7}\!-\!\frac{14}{9}x^{6}\!-\!\frac{1}{9}x^5\!+\!\frac{5}{3}x^{4}\!-\!\frac{1}{9}x^{3}\!-\!\frac{14}{9}x^2\!+\!\frac{1}{3}x\!+\!1$  & $[1, 6, 6]$ \\
\hline
$15$ & ${\rm Alt}_{6}$                & $6$ & $C_{6}$ & $1$ & $x^{2} + \frac{5}{3}x + 1$;\quad  $x^{2} - \frac{7}{4}x + 1$ & $[1, 6, 8]$ \\
\hline
    \end{tabular}
       \caption{Primitive groups degree $n \leq 15$, rank $3$ acting on complex Hadamard matrices.}\label{tab1}
\end{table}

The matrices in the centraliser algebra of a permutation group of rank $3$ have been previously described by Godsil, Chan and  Nu\~{n}ez Ponasso. In addition the matrix of order 11 and two solutions at order 13 have been described by Haagerup \cite{Haagerup}; all are described in the database of complex Hadamard matrices \cite{Karol,CHGuide}. We believe that the remaining matrices are new. In future work, detailed classifications and proofs of inequivalence will be presented. An online database will be maintained at \url{https://github.com/pocathain/CHM}.

\bibliographystyle{plain}
\bibliography{main}
\end{document}